\documentclass[final,leqno]{siamltex}
\pdfoutput=1
\usepackage{amsmath}
\usepackage{amssymb}
\usepackage{array}
\usepackage{graphicx}
\usepackage{tikz,ulem}
\usepackage{dsfont}
\usepackage{subfigure}
\usepackage{graphicx,url,color}
\usepackage{algpseudocode,algorithm}
\newcommand*\boxnode{\node[shape=rectangle,draw,thick] (char)}
\newcommand*\meagrenode{\node[shape=circle,draw, fill=black!100,thick] (char)}
\newcommand*\fatnode{\node[circle,draw, thick] (char)}

\newcommand{\R}{\mathbb{R}}
\newcommand{\I}{\mathbf{I}}

\newcommand{\J}{\mathbf{J}}

\newcommand{\Lb}{\mathbf{L}}

\newtheorem{remark}{Remark}
\title{LIRK-W: Linearly-Implicit Runge-Kutta Methods with Approximate Matrix Factorization}
\author{Paul Tranquilli \footnotemark[2]\ \footnotemark[5] \and Adrian Sandu \footnotemark[3]\ \footnotemark[5]
\and Hong Zhang \footnotemark[4] \ \footnotemark[5]}
\begin{document}
%% COVER PAGE START HERE
\thispagestyle{empty}
\setcounter{page}{0}

\makeatletter
\def\Year#1{%
  \def\yy@##1##2##3##4;{##3##4}%
  \expandafter\yy@#1;
}
\makeatother

\begin{Huge}
\begin{center}
Computational Science Laboratory Technical Report CSL-TR-\Year{\the\year}-{\tt 7} \\
\today
\end{center}
\end{Huge}
\vfil
\begin{huge}
\begin{center}
Paul Tranquilli, Adrian Sandu and  Hong Zhang
\end{center}
\end{huge}

\vfil
\begin{huge}
\begin{it}
\begin{center}
``{\tt LIRK-W: Linearly-implicit Runge-Kutta methods with approximate matrix factorization}''
\end{center}
\end{it}
\end{huge}
\vfil

\begin{large}
\begin{center}
Computational Science Laboratory \\
Computer Science Department \\
Virginia Polytechnic Institute and State University \\
Blacksburg, VA 24060 \\
Phone: (540)-231-2193 \\
Fax: (540)-231-6075 \\ 
Email: \url{sandu@cs.vt.edu} \\
Web: \url{http://csl.cs.vt.edu}
\end{center}
\end{large}

\vspace*{1cm}

\begin{tabular}{ccc}
\includegraphics[width=2.5in]{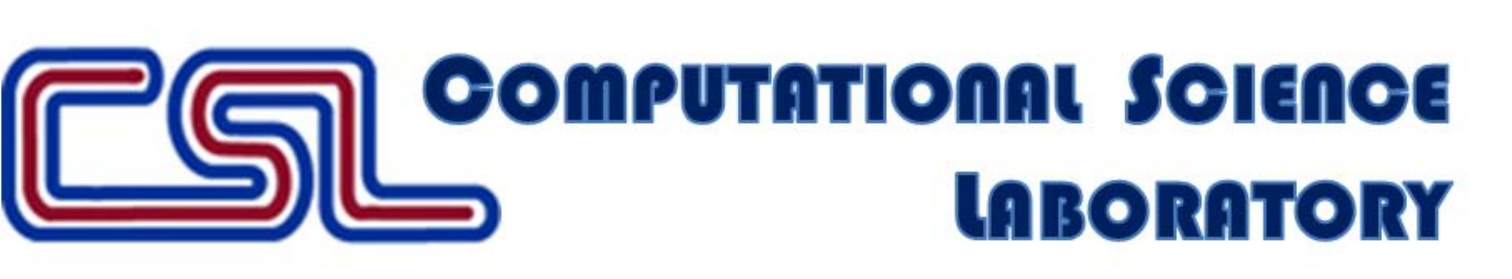}
&\hspace{2.5in}&
\includegraphics[width=2.5in]{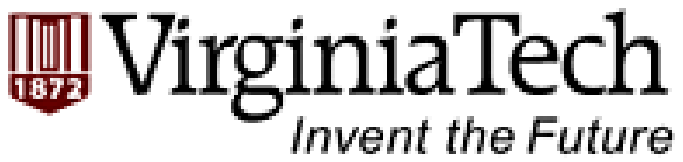} \\
{\bf\em\large Compute the Future} &&\\
\end{tabular}

\newpage
%% COVER PAGE END HERE
%\begin{frontmatter}
%  %\title{LIRK-w Methods}
%  \title{LIRK-W: Linearly-Implicit Runge-Kutta Methods with Approximate Matrix Factorization}
%  \author[csl]{Paul Tranquilli}
%  \ead{ptranq@vt.edu}
%  \address[csl]{Computational Science Laboratory, Department of Computer Science, Virginia Tech.  Blacksburg, Virginia 24060}
%  \author[csl]{Adrian Sandu}
%  \ead{sandu@cs.vt.edu}
%  \author[csl]{Hong Zhang}
%  \ead{zhang@cs.vt.edu}
%
% SIAM
%\title{LIRK-W: Linearly-Implicit Runge-Kutta Methods with Approximate Matrix Factorization}
%\author{Paul Tranquilli \and Adrian Sandu \and Hong Zhang\thanks{Virginia Polytechnic Institute and State University, Computational Science Laboratory, Department of Computer Science, 2202 Kraft Drive, Blacksburg, VA 24060, USA ({\tt sandu@cs.vt.edu})}
%        }
%
%\begin{abstract}
%Abstract goes here.
% \end{abstract}
%
%\end{frontmatter}
%
%\tableofcontents

%\newpage
\maketitle
\renewcommand{\thefootnote}{\fnsymbol{footnote}}
\footnotetext[2]{ptranq@vt.edu}
\footnotetext[3]{sandu@cs.vt.edu}
\footnotetext[4]{zhang@vt.edu}
\footnotetext[5]{Computational Science Laboratory, Department of Computer Science, Virginia Tech.  2202 Kraft Drive, Blacksburg, Virginia 24060.}
\renewcommand{\thefootnote}{\arabic{footnote}}
\begin{abstract}
This paper develops a new class of linearly implicit time integration schemes called Linearly-Implicit Runge-Kutta-W (LIRK-W) methods.  These schemes are based on an implicit-explicit approach which does not require a splitting of the right hand side and allow for arbitrary, time dependent, and stage varying approximations of the linear systems appearing in the method.  Several formulations of LIRK-W schemes, each designed for specific approximation types, and their associated order condition theories are presented.
\end{abstract}
%%%%%%%%%%%%%%%%%%%%%%%%
\section{Introduction}
\label{LIRKW:sec:intro}
%%%%%%%%%%%%%%%%%%%%%%%%
A standard approach to the solution of systems of partial differential equations is the method of lines technique, in which a discretization method such as finite differences, finite volumes, or finite elements is used to approximate derivatives in space to arrive at the semi-discrete initial value problem \eqref{LIRKW:eqn:ode}
\begin{equation}
\label{LIRKW:eqn:ode}
 \frac{dy}{dt} = F(y)\,,~~~ t_0 \leq t \leq t_F\,, \quad y(t_0) = y_0\,; \qquad  y(t), ~~F(y) \in \R^N\,.
\end{equation}
The system \eqref{LIRKW:eqn:ode} can be evolved through time using a time integration scheme to approximate solutions at discrete times $t_n < t_i < t_F$.  For problems with stiff dynamics, or where mesh refinement leads to unfortunate Courant-Friedrichs-Lewy (CFL) numbers, explicit methods may not be suitable.

Implicit methods, such as Backwards Differentation or Runge-Kutta methods improve upon the stability of explicit methods, at the cost of requiring one or more non-linear solves at each timestep.  The approximation of solutions of these nonlinear systems represents the bulk of the computational cost of the time integration process.  

Linearly implicit methods, such as Rosenbrock \cite[section IV.7]{Hairer_book_II} or linearly-implicit Runge-Kutta \cite{Grooms_2011,Calvo_2001,Akrivis_2003_IMEX,Akrivis_2004}, schemes replace the need to solve a nonlinear system with the solution of more computationally efficient linear systems.  Once again, however, the cost of approximating the solution of these linear systems represents a disproportionately large percentage of the overall method cost.  Rosenbrock-W \cite[section IV.7]{Hairer_book_II} \cite{Rang_2005_ROW3} methods attempt to alleviate this burden by allowing for the use of arbitrary approximations, and so permit relatively cheap, and inaccurate, solutions of the linear system while maintaining full order of convergence.  Unfortunately, while these methods permit arbitrary approximations, the order condition theory on which they are built does not account for variations in the method of approximation between stages of the method.  

For this reason, many approximation techniques may not be available in the context of Rosenbrock-W methods.  Primary among them is the use of Krylov based iterative solvers, such as GMRES \cite{Saad}, which computes the solution of several different, nearby linear systems when the iteration procedure is truncated early \cite{Tranquilli2016}.  Similarly, making use of an approximate matrix factorization (AMF) to accelerate the solution of the linear systems also leads to a reduction in the order of the Rosenbrock-W scheme.  Several families of time integration schemes (ROWMAP \cite{Weiner_1997_rowmap} and Rosenbrock-Krylov \cite{Tranquilli_2014_ROK}) have been constructed to couple Krylov solvers and Rosenbrock-W methods to alleviate order reduction in the case of GMRES like approximations.

Here we present a new family of time integration schemes, called linearly-implicit Runge-Kutta-W (LIRK-W) methods, based on an implicit-explicit (IMEX) \cite{Ascher_1997} approach, which allow for arbitrary, time dependent, and stage varying approximations of the linear systems appearing in the method.  The rest of the paper is laid out as follows: In section \ref{LIRKW:sec:lirkw} we motivate, and present, the general form of a LIRK-W method; in section \ref{LIRKW:sec:amf} we review the use of AMF to accelerate the solution of linear systems in the context of PDEs; in section \ref{LIRKW:sec:oc} we present the order condition theory for a LIRK-W method; in section \ref{LIRKW:sec:stability} we give linear stability results for a LIRK-W method; in section \ref{LIRKW:sec:construction} we derive a specific LIRK-W method; and finally in section \ref{LIRKW:sec:numerics} we present numerical results.

%\sfrom 
%%%%%%%%%%%%%%%%%%%%%%%%%%%%%%%%%%%%%%
\section{Linearly-Implicit Runge-Kutta-W (LIRK-W) Methods}
\label{LIRKW:sec:lirkw}
%%%%%%%%%%%%%%%%%%%%%%%%%%%%%%%%%%%%%%

%%%%%%%%%%%%%%%%%%%%%%%%%%%%%%%%%%%%%%
\subsection{Implicit-explicit Runge Kutta methods}
%%%%%%%%%%%%%%%%%%%%%%%%%%%%%%%%%%%%%%

Consider a splitting of the right hand side $F(y)$ of the initial value problem \eqref{LIRKW:eqn:ode}
\begin{equation}
\label{LIRKW:eqn:ode-split}
 \frac{dy}{dt} = F(y) = f(y) + g(y)\,,\quad t_0 \leq t \leq t_F\,, 
\end{equation}
such that $f(y)$ represents slow dynamics, unlikely to impact the stability of the numerical integration, and $g(y)$ contains the fast  dynamics.  An implicit-explicit Runge-Kutta (IMEX-RK) method applies different discretizations to the the two terms, and integrates the non-stiff component $f(y)$ explicitly and stiff component $g(y)$ implicitly \cite{Sandu_2015_GARK}:
\begin{subequations}
\label{LIRKW:eqn:imexrk}
\begin{eqnarray}
\label{LIRKW:eqn:imexrkstage}
Y_i  &=& y_n +   h\, \sum_{j=1}^{i-1}\, a_{i,j}\,   f(Y_j)  +   h\, \sum_{j=1}^i\, \widehat{a}_{i,j}\,   g(Y_j), \quad i=1,\dots,s, \\
\label{LIRKW:eqn:imexrksoln}
y_{n+1} &=&  y_n +   h\, \sum_{j=1}^s\, b_{j}\,   f(Y_j)  +   h\, \sum_{j=1}^s\, \widehat{b}_{j}\,   g(Y_j). 
\end{eqnarray}
\end{subequations}
The fact that only $g(y)$ is treated implicitly has the benefit of reducing the per-timestep cost as compared to implicit Runge-Kutta methods, since one only solves non-linear systems containing $g(y)$ instead of the entire of $F(y)$. In the same time \eqref{LIRKW:eqn:imexrk} has considerably better stability properties than explicit Runge-Kutta schemes since the stiff dynamics is integrated implicitly.  

Unfortunately, the application of \eqref{LIRKW:eqn:imexrk} requires to first partition the system $F(t,y)$ into non-stiff and stiff parts \eqref{LIRKW:eqn:ode-split}, and to provide the Jacobian operator corresponding to the stiff term. These steps can be difficult to achieve for systems implemented in large legacy codes.

%%%%%%%%%%%%%%%%%%%%%%%%%%%%%%%%%%%%%%
\subsection{Arbitrary linear approximations of the stiff term}
%%%%%%%%%%%%%%%%%%%%%%%%%%%%%%%%%%%%%%

Here we propose the new class of Linearly-Implicit Runge-Kutta-W (LIRK-W) time integrators that make use of an alternative partitioning, based on a linear/non-linear splitting of the right hand side operator:
\begin{equation}
\label{LIRKW:eqn:lirkwode}
 \frac{dy}{dt} = \Lb\, y + \left(F(y) - \Lb y\right)\,,~~~ t_0 \leq t \leq t_F\,, \quad y(t_0) = y_n\,; \quad  y(t), F(y) \in \R^N\, , \quad \Lb \in \R^{N \times N}.
\end{equation}
where $\Lb y$ ideally captures the stiff dynamics of $F(y)$.  In this way we will seek to treat implicitly the linear terms $\Lb$ that capture the stiffness of the system, and to treat explicitly the remaining nonlinear part.
The IMEX-RK method \eqref{LIRKW:eqn:imexrk} applied to \eqref{LIRKW:eqn:lirkwode}, after some rearranging of terms, reads: 
\begin{subequations}
\label{LIRKW:eqn:lirkwimex}
\begin{eqnarray}
\left( \I - h\,\widehat{a}_{i,i}\,\Lb\right)\, Y_i  &=& y_n +   h\, \sum_{j=1}^{i-1}\, a_{i,j}\,   F(Y_j)  +  h\,\Lb \sum_{j=1}^{i-1}\,  \left( \widehat{a}_{i,j}-a_{i,j} \right)\,   \, Y_j, \\
y_{n+1} &=&  y_n +   h\, \sum_{j=1}^s\, b_{j}\,   F(Y_j)  +   h\, \Lb\, \sum_{j=1}^s\, \left(\widehat{b}_{j}-b_{j}\right)\,  Y_j. 
\end{eqnarray}
\end{subequations}
Using the notation
\begin{equation*}
\gamma_{i,j} = \left(\widehat{a}_{i,j}-a_{i,j} \right), \quad \textrm{and} \quad g_i = \left(\widehat{b}_{j}-b_{j}\right)
\end{equation*}
in equation \eqref{LIRKW:eqn:lirkwimex} leads to the standard form of a
Linearly-Implicit Runge-Kutta-W (LIRK-W) method:
\begin{subequations}
\label{LIRKW:eqn:lirkw}
\begin{eqnarray}
\label{LIRKW:eqn:lirkwstage}
\left( \I - h\,\gamma_{i,i}\,\Lb\right)\, Y_i  &=& y_n +   h\, \sum_{j=1}^{i-1}\, a_{i,j}\,   F(Y_j)  +  h\,\Lb \sum_{j=1}^{i-1}\, \gamma_{i,j}\,   \, Y_j, \\
\label{LIRKW:eqn:lirkwynew}
y_{n+1} &=&  y_n +   h\, \sum_{j=1}^s\, b_{j}\,   F(Y_j)  +   h\, \Lb\, \sum_{j=1}^s\, g_j\,  Y_j. 
\end{eqnarray}
\end{subequations}

It is clear from equation \eqref{LIRKW:eqn:lirkwode} that the linear operator $\Lb$ can be arbitrary, since collecting terms leads back to the original form of the initial value problem \eqref{LIRKW:eqn:ode}.  Additionally, \eqref{LIRKW:eqn:lirkw} makes exclusive use of the entire right-hand-side vector F(y), there is no splitting necessary. There are several benefits of this framework. First, the stiff terms that are integrated implicitly are linear, so no solutions of nonlinear systems are required. Next, LIRK-W methods will be designed to preserve accuracy for any matrix $\Lb$. The selected $\Lb$ is only used to ensure numerical stability;  its structure is arbitrary and can be chosen to ensure computational efficiency on the hardware at hand. With this in mind,  an ideal choice of the linear operator is one which approximates the stiff dynamics of  the system, $\Lb \approx \partial g(t,y)/\partial y$, in order to improve stability of the numerical integration, and where solutions of linear systems containing $\Lb$ can be computed efficiently.  

LIRK-W methods are similar to Rosenbrock-W schemes, in that they depend only on linear solves and permit arbitrary matrices.  However, as will be discussed more thoroughly in section \ref{LIRKW:sec:norosenbrock}, there are several advantages to the LIRK-W framework.

%%%%%%%%%%%%%%%%%%%%%%%%
\subsection{Approximate Matrix Factorization}
\label{LIRKW:sec:amf}
%%%%%%%%%%%%%%%%%%%%%%%%

Approximate Matrix Factorization (AMF) is often employed to speed up the solution of linear systems required when computing stage vectors \eqref{LIRKW:eqn:imexrkstage} or \eqref{LIRKW:eqn:lirkwstage} of implicit methods:
\begin{equation}
\label{LIRKW:eqn:rksolve}
\left(\I - h \gamma_{i,i} \Lb\right)\, k_i = d_i.
\end{equation}
AMF splits the matrix $\Lb$, usually the Jacobian of the right hand side vector in \eqref{LIRKW:eqn:ode}, into a sum of parts
\begin{equation}
\label{LIRKW:eqn:amfsum}
\Lb = \displaystyle\sum_{r = 1}^R \Lb^{\{r\}},
\end{equation}
and approximates the solution to \eqref{LIRKW:eqn:rksolve} by replacing the matrix as follows:
\begin{equation}
\label{LIRKW:eqn:amfprod}
\begin{split}
\left(\I - h\,\gamma_{i,i}\,\Lb\right) &\approx \left(\I - h\,\gamma_{i,i}\,\widetilde{\Lb}\right) = \displaystyle\prod_{r=1}^R \left(\I - h\,\gamma_{i,i}\,\Lb^{\{r\}}\right), \\
k_i &= \left(\I - h \gamma_{i,i} \Lb\right)^{-1}\, d_i \approx \displaystyle\prod_{r=1}^R \left(\I - h\,\gamma_{i,i}\,\Lb^{\{r\}}\right)^{-1}\, d_i.
\end{split}
\end{equation}
The solution of one large linear system with matrix $\Lb$ is replaced by solving in succession $R$ small linear systems with matrices $\Lb^{\{r\}}$, potentially leading to considerable computational savings.

The approximation \eqref{LIRKW:eqn:amfprod} implicitly defines the matrix $\widetilde{\Lb}$ as
\begin{equation}
\label{LIRKW:eqn:Ltilde}
 \widetilde{\Lb}  = \Lb + \sum_{k=2}^R \, \left(-h \gamma_{i,i}\right)^{k-1} \sum_{1\leq i_1 < i_2 < \dots < i_k \leq R} \Lb_{i_1} \, \Lb_{i_2} \dots \Lb_{i_k} .
\end{equation}
For example, if $\Lb$ is the discrete two-dimensional Laplacian, then in \eqref{LIRKW:eqn:amfsum} the parts $\Lb^{\{1\}}$ and $\Lb^{\{2\}}$ can correspond to derivatives along the $x_1$ and $x_2$ directions respectively.  In this case the AMF approximation corresponds to the alternating directions factorization \cite{Douglas_1962,Peaceman_1955}
\begin{equation*}
\I - h\gamma\widetilde{\Lb} := \left(\I - h\gamma\Lb^{\{1\}}\right)\,\left(\I - h\gamma\Lb^{\{2\}}\right),
\qquad \widetilde{\Lb} = \Lb - h\gamma\Lb^{\{1\}}\Lb^{\{2\}}.
\end{equation*}
In this way it is possible to solve individually much simpler one-dimensional systems corresponding to each individual $\Lb_i$.  An alternating directions factorization is not the only choice possible, but it does motivate many of our design decisions for LIRK-W methods.

%%%%%%%%%%%%%%%%%%%%%%%%
\subsection{Alternative formulations of LIRK-W methods}
\label{LIRKW:sec:formulations}
%%%%%%%%%%%%%%%%%%%%%%%%

We note from \eqref{LIRKW:eqn:Ltilde} that the approximate matrix $\widetilde{\Lb}(h)$ depends on the step size, i.e., depends on time.  The general form \eqref{LIRKW:eqn:lirkw} making use of a time dependent, stage varying matrix $\Lb$ can be interpreted in several ways.  In the most straightforward variation, we can let there be one $\Lb_i$ per stage, and make use of these matrices everywhere that $\Lb$ appears in \eqref{LIRKW:eqn:lirkw}, we will call this a type 1 LIRK-W method  

\begin{subequations}
\label{LIRKW:eqn:lirkwtype1}
\begin{eqnarray}
Y_i  &=& y_n +   h\, \sum_{j=1}^{i-1}\, a_{i,j}\,   F(Y_j)  +  h\,\sum_{j=1}^{i}\, \gamma_{i,j}\,\Lb_j\, Y_j,\\
y_{n+1} &=&  y_n +   h\, \sum_{i=1}^s\, b_{i}\,   F(Y_i)  +   h\, \sum_{i=1}^s\, g_i\, \Lb_i\, Y_i,
\end{eqnarray}
\end{subequations}
with $\Lb_i = \widetilde{\Lb}(t_n + h\gamma_{i,i})$.

Alternatively, we can attempt to interpret \eqref{LIRKW:eqn:lirkw} to account for an AMF form of $\Lb$, where there exists two forms of $\Lb$: equation \eqref{LIRKW:eqn:amfsum}, which has clear advantages for a multiplication by $\Lb$; and equation \eqref{LIRKW:eqn:amfprod}, which has clear advantages when used in the inversion of the left hand side of equation \eqref{LIRKW:eqn:lirkwstage}.  A LIRK-W method of type 2 exploits the natural features of an AMF form of $\Lb$ and has the general form
\begin{subequations}
\label{LIRKW:eqn:lirkwtype2}
\begin{eqnarray}
Y_i  &=& y_n +   h\, \sum_{j=1}^{i-1}\, a_{i,j}\,   F(Y_j)  +  h\,\Lb \sum_{j=1}^{i-1}\, \gamma_{i,j}\,\, Y_j + h\gamma_{i,i}\Lb_i Y_i, \\
y_{n+1} &=&  y_n +   h\, \sum_{j=1}^s\, b_{j}\,   F(Y_j)  +   h\, \Lb\, \sum_{j=1}^s\, g_i\,  Y_j,
\end{eqnarray}
\end{subequations}
where $\Lb_i = \Lb(t_n + h\gamma_{i,i})$.

More variations of \eqref{LIRKW:eqn:lirkw} are possible, and different choices for constructing $\Lb_i$ may motivate alternatives to type 1 or type 2 LIRK-W schemes.

%%%%%%%%%%%%%%%%%%%%%%%%
\subsection{LIRK-W methods are not Rosenbrock-W methods}
\label{LIRKW:sec:norosenbrock}
%%%%%%%%%%%%%%%%%%%%%%%%

As was discussed briefly above, LIRK-W methods share many common traits with a Rosenbrock, or Rosenbrock-W, approach.  Primary among them is that LIRK-W methods require only the solution of linear systems, and permit arbitrary matrices $\Lb$.  The defining characteristic of LIRK-W schemes, in contrast to Rosenbrock-W methods, is the allowance for the use of several different approximations in a single timestep.  Additionally, the arbitrary linear term, $\Lb_i$, may depend on time so that it is possible to directly approximate $\left(\I - h\gamma_{i,i}\Lb_i\right)^{-1}$ without violating assumptions inherent in the order condition framework.

One example of this, is the use of a GMRES like approximation to the linear systems $\left(\I - h\gamma_{i,i}\J\right)x = b_i$.  Early truncation of the GMRES procedure may lead to loss of order for Rosenbrock-W schemes since each linear system is solved using a different subspace, and is essentially the same as solving $s$ systems with $s$ different left-hand-sides.  It is possible to formulate a LIRK-W method for exactly this scenario, one possible approach is:
\begin{subequations}
\label{eqn:lirkwtype3}
\begin{eqnarray}
Y_i  &=& y_n +   h\, \sum_{j=1}^{i-1}\, a_{i,j}\,   F(Y_j)  +  h\,\Lb_i \sum_{j=1}^{i}\, \gamma_{i,j}\,\, Y_j  \\
y_{n+1} &=&  y_n +   h\, \sum_{j=1}^s\, b_{j}\,   F(Y_j)  +   h\, \Lb\, \sum_{j=1}^s\, g_i\,  Y_j,
\end{eqnarray}
\end{subequations}
which we will refer to as type 3 LIRK-W method.

%\sto 

%%%%%%%%%%%%%%%%%%%%%%%%
\section{Order Conditions for Linearly-Implicit Runge-Kutta-W Methods}
\label{LIRKW:sec:oc}
%%%%%%%%%%%%%%%%%%%%%%%%

We construct classical order conditions for the various types of LIRK-W methods by matching the Taylor series expansion of the exact solution of equation \eqref{LIRKW:eqn:ode} with the Taylor series expansion of the numerical solutions.  To do so we make use of Butcher trees \cite[section II.2]{Hairer_book_I}, with suitable modifications to handle the particular aspects of LIRK-W methods. These modifications lead to the new family of $LW-$trees, which we discuss next.

%%%%%%%%%%%%%%%%%%%%%%%%
\subsection{LW-trees}
\label{LIRKW:sec:trees}
%%%%%%%%%%%%%%%%%%%%%%%%

These trees contain three different colored vertices: meagre vertices are filled and represent an appearance of $F(y)$ and its derivatives, fat vertices are empty and represent an appearance of the linear term $\Lb$, and finally square vertices represent a differentiation in time of the fat node they terminate in.  We refer to the set of these trees as $LW_i$-trees, where $i$ denotes the type of LIRK-W method being discussed.

For LIRK-W methods of all types, we make use of the standard notation that the tree $\left[\tau_1, \dots, \tau_m\right]_{\bullet}$ is constructed by attaching the subtrees $\tau_1, \dots, \tau_m$ to a meagre node as its root.  Similarly, the $\left[ \tau \right]_\circ$ is constructed by attaching the subtree $\tau$ to a fat node, and new to this manuscript is the notation that $\left[ \tau \right]_{\theta_p}$ is constructed by attaching the subtree $\tau$ to a fat node prepended by $p$ square nodes.  Figure \ref{LIRKW:fig:treenotation} shows how trees can be constructed from sub-trees.

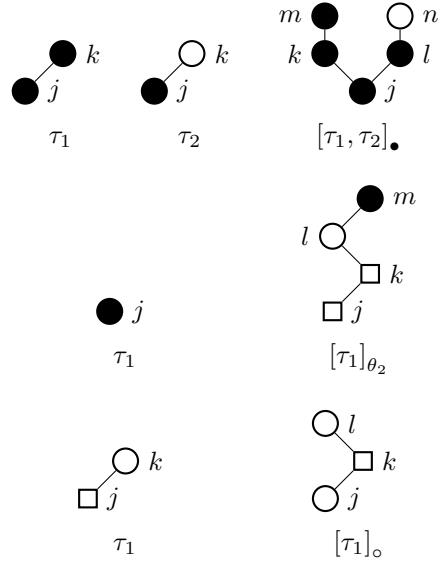
\begin{figure}[H]
\begin{center}
\begin{tabular}{ccc}
\begin{tikzpicture}[scale=.5]
      \meagrenode (j) at (0,0) [label=right:$j$] {};
      \meagrenode (k) at (1,1) [label=right:$k$] {};
      \draw[-] (j) -- (k);
  \end{tikzpicture} & 
\begin{tikzpicture}[scale=.5]
      \meagrenode (j) at (0,0) [label=right:$j$] {};
      \fatnode (k) at (1,1) [label=right:$k$] {};
      \draw[-] (j) -- (k);
  \end{tikzpicture} & 
\begin{tikzpicture}[scale=.5]
      \meagrenode (j) at (1,0) [label=right:$j$] {};
      \meagrenode (k) at (0,1) [label=left:$k$] {};
      \meagrenode (l) at (2,1) [label=right:$l$] {};
      \meagrenode (m) at (0,2) [label=left:$m$] {};
      \fatnode    (n) at (2,2) [label=right:$n$] {};
      \draw[-] (j) -- (k);
      \draw[-] (j) -- (l);
      \draw[-] (k) -- (m);
      \draw[-] (l) -- (n);
  \end{tikzpicture} \\
$\tau_1$ & $\tau_2$ & $\left[\tau_1, \tau_2\right]_\bullet$ \\
& &  \\
\multicolumn{2}{c}{\begin{tikzpicture}[scale=.5]
      \meagrenode (j) at (0,0) [label=right:$j$] {};
        \end{tikzpicture}} & \begin{tikzpicture}[scale=.5]
      \boxnode (j) at (0,0) [label=right:$j$] {};
      \boxnode (k) at (1,1) [label=right:$k$] {};
      \fatnode (l) at (0,2) [label=left:$l$] {};
      \meagrenode(m) at (1,3) [label=right:$m$] {};
      \draw[-] (j) -- (k);
      \draw[-] (k) -- (l);
      \draw[-] (l) -- (m);
  \end{tikzpicture} \\
\multicolumn{2}{c}{$\tau_1$} & $\left[\tau_1\right]_{\theta_2}$\\
& & \\
\multicolumn{2}{c}{\begin{tikzpicture}[scale=.5]
      \boxnode (j) at (0,0) [label=right:$j$] {};
      \fatnode (k) at (1,1) [label=right:$k$] {};
      \draw[-] (j) -- (k);
        \end{tikzpicture}} & \begin{tikzpicture}[scale=.5]
      \fatnode (j) at (0,0) [label=right:$j$] {};
      \boxnode (k) at (1,1) [label=right:$k$] {};
      \fatnode (l) at (0,2) [label=right:$l$] {};
      \draw[-] (j) -- (k);
      \draw[-] (k) -- (l);
  \end{tikzpicture} \\
\multicolumn{2}{c}{$\tau_1$} & $\left[\tau_1\right]_{\circ}$
\end{tabular}
\caption{Illustration of the recursive definition of $LW$-trees.}
\label{LIRKW:fig:treenotation}
\end{center}
\end{figure}

To derive the expansion of the numerical solution for type 1 and 2 methods we make use of Fa{\`a} di Bruno's formula \cite[section II.2]{Hairer_book_I} to compute general high-order derivatives of the initial value problem right-hand-side vector $F(y)$
\begin{equation}
\label{LIRKW:eqn:faadibruno}
\left.\left(F(Y_j)\right)^{(q-1)}\right|_{h=0} = \displaystyle\sum_{m>=1} \frac{\partial^m f}{\partial y^m}\left(Y_j^{(\mu_1)}, \dots, Y_j^{(\mu_m)}\right), \quad \mu_1 + \dots + \mu_m = q-1.
\end{equation}
%

%%%%%%%%%%%%%%%%%%%%%%%%%%%%%%%%%%%%%%%%
\subsection{Order conditions for LIRK-W methods of type 1}
\label{LIRKW:sec:order-type-1}
%%%%%%%%%%%%%%%%%%%%%%%%%%%%%%%%%%%%%%%%

Here we derive order conditions for LIRK-W methods of type 1 \eqref{LIRKW:eqn:lirkwtype1}.

We can make use of the expansion \eqref{LIRKW:eqn:faadibruno} to construct derivatives of the type 1 LIRK-W scheme to arrive at the result that
\begin{subequations}
\label{LIRKW:eqn:lirkwexpansiontype1}
\begin{equation}
\left.\left(Y_i\right)^{(q)}\right|_{h=0} = q \displaystyle\sum_{j=1}^{i-1}a_{i,j} \displaystyle\sum_{m\geq 1}\frac{\partial^m F}{\partial y^m}\left(Y_j^{(\mu_1)}, \dots, Y_j^{(\mu_m)}\right) + q\displaystyle\sum_{j=1}^i \gamma_{i,j} \displaystyle\sum_{k=0}^{q-1}{q-1 \choose k}\gamma_{j,j}^k\left(\frac{d^k}{dt^k}\Lb\right)Y_j^{(q-1-k)}
\end{equation}
\begin{equation}
\left.\left(y_{n+1}\right)^{(q)}\right|_{h=0} = q \displaystyle\sum_{i=1}^{s}b_{i} \displaystyle\sum_{m\geq 1}\frac{\partial^m F}{\partial y^m}\left(Y_i^{(\mu_1)}, \dots, Y_i^{(\mu_m)}\right) + q\displaystyle\sum_{i=1}^s g_{i} \displaystyle\sum_{k=0}^{q-1}{q-1 \choose k}\gamma_{i,i}^k\left(\frac{d^k}{dt^k}\Lb\right)Y_i^{(q-1-k)}
\end{equation}
\end{subequations}
Leading to the definition for the $LW_1$ trees
\[ 
LW_1 = \left\{ \begin{array}{cl}N_3\textrm{-trees:} & \textrm{fat vertices are singly branched, and} \\
				 & \textrm{square vertices are singly branched} \\ 
				 & \quad \quad\textrm{with square or fat children}  \end{array} \right\} 
\]

Figures \ref{LIRKW:fig:lirkwtrees1} and \ref{LIRKW:fig:lirkwtrees2}, give the $LW_1$-trees, $\tau_i$, representing the elementary differentials, $f(\tau_i)$, present in the Taylor expansion of the exact and numerical solutions, as well as the coefficients of these terms, $\Phi^1(\tau_i)$ and $P(\tau_i)$ respectively.

\begin{figure}[htp]
\begin{center}
\def\arraystretch{1.5}
\footnotesize
\begin{tabular}{|c|c|c|c|c|c|}
\hline
$i$ & 1 & 2 & 3 & 4 & 5\\
\hline
$\tau_i$ & 
\begin{tikzpicture}[scale=.5]
      \meagrenode (j) at (0,0) [label=right:$j$] {};
  \end{tikzpicture} & 
\begin{tikzpicture}[scale=.5]
      \fatnode (j) at (0,0) [label=right:$j$] {};
  \end{tikzpicture} & 
\begin{tikzpicture}[scale=.5]
      \meagrenode (j) at (0,0) [label=right:$j$] {};
      \meagrenode (k) at (1,1) [label=right:$k$] {};
      \draw[-] (j) -- (k);
  \end{tikzpicture} &
\begin{tikzpicture}[scale=.5]
      \meagrenode (j) at (0,0) [label=right:$j$] {};
      \fatnode (k) at (1,1) [label=right:$k$] {};
      \draw[-] (j) -- (k);
  \end{tikzpicture} &
\begin{tikzpicture}[scale=.5]
      \fatnode (j) at (0,0) [label=right:$j$] {};
      \meagrenode (k) at (1,1) [label=right:$k$] {};
      \draw[-] (j) -- (k);
  \end{tikzpicture} \\
\hline
$f(\tau_i)$ & $f^J$ & $\Lb_{JK}y^K$  & $f^J_Kf^K$ & $f^J_K\Lb_{KL}y^L$ & $\Lb_{JK}f^K$ \\
\hline
$\Phi^1(\tau_i)$ & $ b_j$ & $ g_j$ &  $b_j a_{j,k}$ & $b_j \gamma_{j,k}$ & $g_j a_{j,k}$\\
\hline
$\Phi^2(\tau_i)$ & $ b_j$ & $ g_j$ &  $b_j a_{j,k}$ & $b_j \gamma_{j,k}$ & $g_j a_{j,k}$\\
\hline
$\Phi^3(\tau_i)$ & $ b_j$ & $ g_j$ &  $b_j a_{j,k}$ & $b_j \gamma_{j,k}$ & $g_j a_{j,k}$\\
\hline
$P(\tau_i)$ & 1 & 0 & 1/2 & 0 & 0 \\
\hline
\hline
$i$ & 6 & 7 & 8 & 9 & 10\\
\hline
$\tau_i$ & 
\begin{tikzpicture}[scale=.5]
      \fatnode (j) at (0,0) [label=right:$j$] {};
      \fatnode (k) at (1,1) [label=right:$k$] {};
      \draw[-] (j) -- (k);
  \end{tikzpicture} &
\begin{tikzpicture}[scale=.5]
      \boxnode (j) at (0,0) [label=right:$j$] {};
      \fatnode (k) at (1,1) [label=right:$k$] {};
      \draw[-] (j) -- (k);
  \end{tikzpicture} &
\begin{tikzpicture}[scale=.5]
      \meagrenode (j) at (1,0) [label=right:$j$] {};
      \meagrenode (k) at (0,1) [label=left:$k$] {};
      \meagrenode (l) at (2,1) [label=right:$l$] {};
      \draw[-] (j) -- (k);
      \draw[-] (j) -- (l);
  \end{tikzpicture} &
\begin{tikzpicture}[scale=.5]
      \meagrenode (j) at (1,0) [label=right:$j$] {};
      \fatnode (k) at (0,1) [label=left:$k$] {};
      \meagrenode (l) at (2,1) [label=right:$l$] {};
      \draw[-] (j) -- (k);
      \draw[-] (j) -- (l);
  \end{tikzpicture}  &
\begin{tikzpicture}[scale=.5]
      \meagrenode (j) at (1,0) [label=right:$j$] {};
      \fatnode (k) at (0,1) [label=left:$k$] {};
      \fatnode (l) at (2,1) [label=right:$l$] {};
      \draw[-] (j) -- (k);
      \draw[-] (j) -- (l);
  \end{tikzpicture}  \\
\hline
$f(\tau_i)$ & $\Lb_{JK}\Lb_{KL}y^L$ & $\Lb'_{JK}y^K$ & $f^J_{KL}f^Kf^L$ & $f^J_{KL}\Lb_{KM}y^Mf^L$ &   $f^J_{KL}\Lb_{KM}y^M\Lb_{LN}y^N$ \\
\hline
$\Phi^1(\tau_i)$ & $g_j \gamma_{j,k}$ & $g_k\gamma_{k,k}$ & $b_ja_{j,k}a_{j,l}$ & $b_j \gamma_{j,k} a_{j,l}$ & $b_j \gamma_{j,k} \gamma_{j,l}$ \\
\hline
$\Phi^2(\tau_i)$ & $g_j \gamma_{j,k}$ & N/A & $b_ja_{j,k}a_{j,l}$ & $b_j \gamma_{j,k} a_{j,l}$ & $b_j \gamma_{j,k} \gamma_{j,l}$ \\
\hline
$\Phi^3(\tau_i)$ & $g_j \gamma_{j,k}$ & N/A & $b_ja_{j,k}a_{j,l}$ & $b_j \gamma_{j,k} a_{j,l}$ & $b_j \gamma_{j,k} \gamma_{j,l}$ \\
\hline
$P(\tau_i)$ & 0 & 0 & 1/3 & 0 & 0 \\
\hline
\hline
$i$ & 11 & 12 & 13 & 14 & 15\\
\hline
$\tau_i$ & 
\begin{tikzpicture}[scale=.5]
      \meagrenode (j) at (0,0) [label=right:$j$] {};
      \meagrenode (k) at (1,1) [label=right:$k$] {};
      \meagrenode (l) at (0,2) [label=left:$l$] {};
      \draw[-] (j) -- (k);
      \draw[-] (k) -- (l);
  \end{tikzpicture} &
\begin{tikzpicture}[scale=.5]
      \meagrenode (j) at (0,0) [label=right:$j$] {};
      \meagrenode (k) at (1,1) [label=right:$k$] {};
      \fatnode (l) at (0,2) [label=left:$l$] {};
      \draw[-] (j) -- (k);
      \draw[-] (k) -- (l);
  \end{tikzpicture} &
\begin{tikzpicture}[scale=.5]
      \meagrenode (j) at (0,0) [label=right:$j$] {};
      \fatnode (k) at (1,1) [label=right:$k$] {};
      \fatnode (l) at (0,2) [label=left:$l$] {};
      \draw[-] (j) -- (k);
      \draw[-] (k) -- (l);
  \end{tikzpicture} &
\begin{tikzpicture}[scale=.5]
      \meagrenode (j) at (0,0) [label=right:$j$] {};
      \fatnode (k) at (1,1) [label=right:$k$] {};
      \meagrenode (l) at (0,2) [label=left:$l$] {};
      \draw[-] (j) -- (k);
      \draw[-] (k) -- (l);
  \end{tikzpicture} &
\begin{tikzpicture}[scale=.5]
      \fatnode (j) at (0,0) [label=right:$j$] {};
      \fatnode (k) at (1,1) [label=right:$k$] {};
      \meagrenode (l) at (0,2) [label=left:$l$] {};
      \draw[-] (j) -- (k);
      \draw[-] (k) -- (l);
  \end{tikzpicture} \\
\hline
$f(\tau_i)$ & $f^J_Kf^K_Lf^L$ & $f^J_Kf^K_L\Lb_{LM}y^M$ & $f^J_K\Lb_{KL}\Lb_{LM}y^M$ & $f^J_K \Lb_{KL}f^L$ & $ \Lb_{JK}\Lb_{KL}f^L$ \\
\hline
$\Phi^1(\tau_i)$ & $b_j a_{j,k} a_{k,l}$ & $b_j a_{j,k} \gamma_{k,l}$ & $b_j \gamma_{j,k} \gamma_{k,l}$ & $b_j \gamma_{j,k} a_{k,l} $ & $g_j \gamma_{j,k} a_{k,l}$ \\
\hline
$\Phi^2(\tau_i)$ & $b_j a_{j,k} a_{k,l}$ & $b_j a_{j,k} \gamma_{k,l}$ & $b_j \gamma_{j,k} \gamma_{k,l}$ & $b_j \gamma_{j,k} a_{k,l} $ & $g_j \gamma_{j,k} a_{k,l}$ \\
\hline
$\Phi^3(\tau_i)$ & $b_j a_{j,k} a_{k,l}$ & $b_j a_{j,k} \gamma_{k,l}$ & $b_j \gamma_{j,k} \gamma_{k,l}$ & $b_j \gamma_{j,k} a_{k,l} $ & $g_j \gamma_{j,k} a_{k,l}$ \\
\hline
$P(\tau_i)$ & 1/6 & 0 & 0 & 0 & 0 \\
\hline
\end{tabular}
\caption{Trees and order conditions for LIRK-W methods up to order three.}
\label{LIRKW:fig:lirkwtrees1}
\end{center}
\end{figure}

\begin{figure}[htp]
\begin{center}
\def\arraystretch{1.5}
\footnotesize
\begin{tabular}{|c|c|c|c|c|}
\hline
$i$ & 16 & 17 & 18 & 19  \\
\hline
$\tau_i$ & 
\begin{tikzpicture}[scale=.5]
      \fatnode (j) at (0,0) [label=right:$j$] {};
      \meagrenode (k) at (1,1) [label=right:$k$] {};
      \meagrenode (l) at (0,2) [label=left:$l$] {};
      \draw[-] (j) -- (k);
      \draw[-] (k) -- (l);
  \end{tikzpicture} &
\begin{tikzpicture}[scale=.5]
      \fatnode (j) at (0,0) [label=right:$j$] {};
      \meagrenode (k) at (1,1) [label=right:$k$] {};
      \fatnode (l) at (0,2) [label=left:$l$] {};
      \draw[-] (j) -- (k);
      \draw[-] (k) -- (l);
  \end{tikzpicture} &
\begin{tikzpicture}[scale=.5]
      \fatnode (j) at (0,0) [label=right:$j$] {};
      \fatnode (k) at (1,1) [label=right:$k$] {};
      \fatnode (l) at (0,2) [label=left:$l$] {};
      \draw[-] (j) -- (k);
      \draw[-] (k) -- (l);
  \end{tikzpicture} &
\begin{tikzpicture}[scale=.5]
      \boxnode (j) at (0,0) [label=right:$j$] {};
      \boxnode (k) at (1,1) [label=right:$k$] {};
      \fatnode (l) at (0,2) [label=left:$l$] {};
      \draw[-] (j) -- (k);
      \draw[-] (k) -- (l);
  \end{tikzpicture} \\ 
\hline
$f(\tau_i)$ & $\Lb_{JK}f^K_Lf^L$ & $\Lb_{JK}f^K_L\Lb_{LM}y^M$ & $ \Lb_{JK}\Lb_{KL}\Lb_{LM}y^M$ & $\Lb''_{JM}y^M$  \\
\hline
$\Phi^1(\tau_i)$ & $g_j a_{j,k} a_{k,l}$ & $g_j a_{j,k} \gamma_{k,l}$ & $g_j \gamma_{j,k} \gamma_{k,l}$ & $g_l \gamma_{l,l} \gamma_{l,l}$ \\
\hline
$\Phi^2(\tau_i)$ & $g_j a_{j,k} a_{k,l}$ & $g_j a_{j,k} \gamma_{k,l}$ & $g_j \gamma_{j,k} \gamma_{k,l}$ & N/A \\
\hline
$\Phi^3(\tau_i)$ & $g_j a_{j,k} a_{k,l}$ & $g_j a_{j,k} \gamma_{k,l}$ & $g_j \gamma_{j,k} \gamma_{k,l}$ & N/A \\
\hline
$P(\tau_i)$ & 0 & 0 & 0 & 0 \\
\hline
\hline
$i$ & 20 & 21 & 22 & 23   \\
\hline
$\tau_i$ & 
\begin{tikzpicture}[scale=.5]
      \fatnode (j) at (0,0) [label=right:$j$] {};
      \boxnode (k) at (1,1) [label=right:$k$] {};
      \fatnode (l) at (0,2) [label=left:$l$] {};
      \draw[-] (j) -- (k);
      \draw[-] (k) -- (l);
  \end{tikzpicture} &
\begin{tikzpicture}[scale=.5]
      \boxnode (j) at (0,0) [label=right:$j$] {};
      \fatnode (k) at (1,1) [label=right:$k$] {};
      \fatnode (l) at (0,2) [label=left:$l$] {};
      \draw[-] (j) -- (k);
      \draw[-] (k) -- (l);
  \end{tikzpicture} &
\begin{tikzpicture}[scale=.5]
      \boxnode (j) at (0,0) [label=right:$j$] {};
      \fatnode (k) at (1,1) [label=right:$k$] {};
      \meagrenode (l) at (0,2) [label=left:$l$] {};
      \draw[-] (j) -- (k);
      \draw[-] (k) -- (l);
  \end{tikzpicture} &
\begin{tikzpicture}[scale=.5]
      \meagrenode (j) at (0,0) [label=right:$j$] {};
      \boxnode (k) at (1,1) [label=right:$k$] {};
      \fatnode (l) at (0,2) [label=left:$l$] {};
      \draw[-] (j) -- (k);
      \draw[-] (k) -- (l);
  \end{tikzpicture}  \\
\hline
$f(\tau_i)$ & $\Lb_{JK}\Lb'_{KM}y^M$ & $\Lb'_{JL}\Lb_{LM}y^M$ & $ \Lb'_{JL}f^L $ & $f^J_K\Lb'_{KM}y^M $ \\
\hline
$\Phi^1(\tau_i)$ & $g_j \gamma_{j,l} \gamma_{l,l}$ & $g_k \gamma_{k,k} \gamma_{k,l}$ & $g_k\gamma_{k,k}a_{k,l}$ & $b_j \gamma_{j,l}\gamma_{l,l}$  \\
\hline
$\Phi^2(\tau_i)$ & $g_j \gamma_{j,j}\gamma_{j,j}$ & N/A & N/A & $b_j \gamma_{j,j}\gamma_{j,j}$  \\
\hline
$\Phi^3(\tau_i)$ & $g_j \gamma_{j,j}\gamma_{j,l}$ & N/A & N/A & $b_j \gamma_{j,j}\gamma_{j,l}$  \\
\hline
$P(\tau_i)$ & 0 & 0 & 0 & 0  \\
\hline
\end{tabular}
\caption{Trees and order conditions for LIRK-W methods up to order three (Continued).}
\label{LIRKW:fig:lirkwtrees2}
\end{center}
\end{figure}

The recursive formulas in equation \eqref{LIRKW:eqn:recursiveordertype1} map the visual representation of a given $LW_1$-tree to the left hand side of the order condition, $\Phi_j(\tau)$, corresponding to that tree.
\begin{subequations}
\label{LIRKW:eqn:recursiveordertype1}
\begin{eqnarray}
\label{LIRKW:eqn:reorderphitype1}
\Phi^1_j(\tau) & = & \left\{ 
\begin{array}{lcrcl} 
b_j\bar{\Phi}^1_{j}(\tau_1)\cdots\bar{\Phi}^1_{j}(\tau_m) & \textrm{if} & \tau & = & \left[\tau_1, \dots, \tau_m\right]_{\bullet}, \\
g_j \bar{\Phi}^1_{j}(\tau_1) & \textrm{if} & \tau & = & \left[\tau_1\right]_\circ, \\
g_j \gamma_{j,j}^p \bar{\Phi}^1_{j}(\tau_1) & \textrm{if} & \tau & = & \left[\tau_1\right]_{\theta_p}, \\
b_j & \textrm{if} & \tau & = & \left[ \; \right]_\bullet, \\
g_j & \textrm{if} & \tau & = & \left[ \; \right]_\circ,
\end{array} \right. \\
\label{LIRKW:eqn:reorderphibartype1}
\bar{\Phi}^1_j(\tau) & = & \left\{
\begin{array}{lcrcl}
a_{j,k} \bar{\Phi}^1_{k}(\tau_1) \cdots \bar{\Phi}^1_{k}(\tau_m) & \textrm{if} & \tau & = & \left[\tau_1, \dots, \tau_m\right]_\bullet, \\
\gamma_{j,k} \bar{\Phi}^1_{k}(\tau_1) & \textrm{if} & \tau & = & \left[ \tau_1 \right]_\circ, \\
\gamma_{j,k}\gamma_{k,k}^p \bar{\Phi}^1_j(\tau_1) & \textrm{if} & \tau & = & \left[ \tau_1 \right]_{\theta_p}, \\
a_{j,k} & \textrm{if} & \tau & = & \left[ \; \right]_\bullet, \\
\gamma_{j,k} & \textrm{if} & \tau & = & \left[ \; \right]_\circ, \\
\end{array} \right. \\
\end{eqnarray}
\end{subequations}
Similarly, the recursive formula in equation \eqref{LIRKW:eqn:recursivedifferential} maps the visual representation of an $LW_1$-tree to its corresponding elementary differential in the Taylor series expansion of the numerical solution computed using a LIRK-W scheme \eqref{LIRKW:eqn:lirkwtype1}.
\begin{equation}
\label{LIRKW:eqn:recursivedifferential}
F^J(\tau)(y) = \left\{ 
\begin{array}{clcrcl}
\displaystyle\sum_{K_1, \dots, K_m} & f^J_{K_1, \dots, K_M}(y)\cdot\left(F^{K_1}(\tau_1) \cdots F^{K_m}(\tau_m)\right)(y) & \textrm{if} & \tau & = & \left[ \tau_1, \dots, \tau_m\right]_\bullet \\
\displaystyle\sum_{K} & \Lb_{JK}F^K(\tau_1)(y) & \textrm{if} & \tau & = & \left[\tau_1\right]_\circ \\
\displaystyle\sum_{K} & \left(\frac{d^p}{dt^p}\Lb_{JK}\right) F^K(\tau_1)(y) & \textrm{if} & \tau & = & \left[\tau_1\right]_{\theta_p} \\
\displaystyle\sum_{K} & \Lb_{JK}y^K & \textrm{if} &  \tau & = & []_{\circ} \\
\displaystyle\sum_{K} & \left(\frac{d^p}{dt^p}\Lb_{JK}\right)y^K & \textrm{if} & \tau & = & []_{\theta_p}
\end{array} \right.
\end{equation}

\begin{remark}
We note that the Taylor series expansion of the exact solution will be comprised of only terms containing the function $F(t,y)$ and its derivatives, and so these terms will be represented by trees containing only meagre nodes.  Alternatively, the Taylor series expansion of the numerical solution will be comprised of terms containing the function $F(t,y)$ and its derivatives, as well as terms containing $\Lb$ and its derivatives.  For this reason, the order conditions corresponding to trees containing fat and square nodes will be set to zero, to eliminate any contribution of these terms to the error of the LIRK-W method.
\end{remark}
\begin{theorem}[Order conditions for LIRK-W methods of type 1]\label{LIRKW:thm:LIRKWtype1-conditions}
A LIRK-W method of type 1 has order $p$ iff the following order conditions hold:
\label{LIRKW:eqn:LIRKWtype1-conditions}
\begin{eqnarray}
\label{LIRKW:eqn:LIRKtype1-condition-T}
\sum_j \Phi^1_j(\tau) = \frac{1}{\gamma(\tau)} \quad \forall\, \tau \in T ~~ \mbox{with } \rho(\tau) \le p\,, \\
\label{LIRKW:eqn:LIRKtype1-condition-TW}
\sum_j \Phi^1_j(\tau) = 0 \quad \forall\, \tau \in LW_1\backslash T ~~ \mbox{with } \rho(\tau) \le p\,.
\end{eqnarray}
Here $\rho(t)$ is the number of vertices of the tree $t$, and $\gamma(t)$ is the ``product of $\rho(t)$ and all orders of the trees which appear, if the roots, one after another, are removed from $t$'' \cite[Section II.2]{Hairer_book_I}.
\end{theorem}
\begin{proof}
 The proof follows from our discussion, equation \eqref{LIRKW:eqn:lirkwexpansiontype1} and the order conditions of Rosenbrock-W methods \cite[Section IV.7]{Hairer_book_II}.
\end{proof}
%

%%%%%%%%%%%%%%%%%%%%%%%%%%%%%%%%%%%%
\subsection{Order conditions for LIRK-W methods of type 2}
\label{LIRKW:sec:order-type-2}
%%%%%%%%%%%%%%%%%%%%%%%%%%%%%%%%%%%%

Here we derive order conditions for  LIRK-W methods of type 2 \eqref{LIRKW:eqn:lirkwtype2}.

We once again make use of the expansion \eqref{LIRKW:eqn:faadibruno} to construct derivatives of the type 2 LIRK-W scheme to arrive at the result that
\begin{subequations}
\label{LIRKW:eqn:lirkwexpansiontype2}
\begin{multline}
\left.\left(Y_i\right)^{(q)}\right|_{h=0} = q \displaystyle\sum_{j=1}^{i-1}a_{i,j}\displaystyle\sum_{m>=1} \frac{\partial^m f}{\partial y^m}\left(Y_j^{(\mu_1)}, \dots, Y_j^{(\mu_m)}\right) + q\Lb\displaystyle\sum_{j=1}^s \left(Y_j\right)^{(q-1)} \\
 + q \displaystyle\sum_{k=0}^{q-1} {q-1 \choose k}\gamma_{i,i}^{k+1}\left(\frac{d^k}{dt^k}\Lb\right)\left(Y_i\right)^{(q-1-k)}
\end{multline}
\begin{equation}
\left.\left(y_{n+1}\right)^{(q)}\right|_{h=0} = q \displaystyle\sum_{j=1}^s b_j \displaystyle\sum_{m>=1} \frac{\partial^m f}{\partial y^m}\left(Y_j^{(\mu_1)}, \dots, Y_j^{(\mu_m)}\right) + q\Lb \displaystyle\sum_{j=1}^s g_j \left(Y_j\right)^{(q-1)}
\end{equation}
\end{subequations}
Leading to the definition of $LW_2$-trees
\[ 
LW_2 = \left\{ \begin{array}{cl}N_3\textrm{-trees:} & \textrm{fat vertices are singly branched, and} \\
				 & \textrm{no tree has a square root, and} \\
				 & \textrm{square vertices are singly branched} \\ 
				 & \quad \quad\textrm{with square or fat children.} \end{array} \right\} 
\]

Figures \ref{LIRKW:fig:lirkwtrees1} and \ref{LIRKW:fig:lirkwtrees2}, give the $LW_2$-trees, $\tau_i$, representing the elementary differentials, $f(\tau_i)$, present in the Taylor expansion of the exact and numerical solutions, as well as the coefficients of these terms, $\Phi^2(\tau_i)$ and $P(\tau_i)$ respectively.

\begin{remark}
$LW_2$-trees are a subset of $LW_1$-trees, and so figures \ref{LIRKW:fig:lirkwtrees1} and \ref{LIRKW:fig:lirkwtrees2} contains trees which are not present in an expansion of the numerical solution of a LIRK-W method of type 2.  These trees can be identified by the presence of a square root, and the corresponding entry for $\Phi^2(\tau_i)$ being listed as ``N/A''.  In addition, $\Phi^1(\tau_i)$ and $\Phi^2(\tau_i)$ are identical, except in the case where $\tau_i$ contains a square node.
\end{remark}

The recursive formulas in equation \eqref{LIRKW:eqn:recursiveordertype2} map the visual representation of a given $LW$-tree to the left hand side of the order condition, $\Phi_j(\tau)$, corresponding to that tree.
\begin{subequations}
\label{LIRKW:eqn:recursiveordertype2}
\begin{eqnarray}
\label{LIRKW:eqn:reorderphitype2}
\Phi^2_j(\tau) & = & \left\{ 
\begin{array}{lcrcl} 
b_j\bar{\Phi}^2_{j}(\tau_1)\cdots\bar{\Phi}^2_{j}(\tau_m) & \textrm{if} & \tau & = & \left[\tau_1, \dots, \tau_m\right]_{\bullet}, \\
g_j \bar{\Phi}^2_{j}(\tau_1) & \textrm{if} & \tau & = & \left[\tau_1\right]_\circ, \\
b_j & \textrm{if} & \tau & = & \left[ \; \right]_\bullet, \\
g_j & \textrm{if} & \tau & = & \left[ \; \right]_\circ,
\end{array} \right. \\
\label{LIRKW:eqn:reorderphibartype2}
\bar{\Phi}^2_j(\tau) & = & \left\{
\begin{array}{lcrcl}
a_{j,k} \bar{\Phi}^2_{k}(\tau_1) \cdots \bar{\Phi}^2_{k}(\tau_m) & \textrm{if} & \tau & = & \left[\tau_1, \dots, \tau_m\right]_\bullet, \\
\gamma_{j,k} \bar{\Phi}^2_{k}(\tau_1) & \textrm{if} & \tau & = & \left[ \tau_1 \right]_\circ, \\
\gamma_{jj}^{p+1} \bar{\Phi}^2_j(\tau_1) & \textrm{if} & \tau & = & \left[ \tau_1 \right]_{\theta_p}, \\
a_{j,k} & \textrm{if} & \tau & = & \left[ \; \right]_\bullet, \\
\gamma_{j,k} & \textrm{if} & \tau & = & \left[ \; \right]_\circ, \\
\end{array} \right.
\end{eqnarray}
\end{subequations}
$LW_1$ and $LW_2$ map from trees to elementary differentials in the same way, so that once again equation \eqref{LIRKW:eqn:recursivedifferential} maps the visual representation of an $LW_2$-tree to its corresponding elementary differential in the Taylor series expansion of the numerical solution using the LIRK-W scheme of type 2 in equation \eqref{LIRKW:eqn:lirkwtype2}.

Making use of the same logic as for LIRK-W methods of type-1 we have that for LIRK-W methods of type 2
\begin{theorem}[Order conditions for LIRK-W methods of type 2]\label{LIRKW:thm:LIRKWtype2-conditions}
A LIRK-W method of type 1 has order $p$ iff the following order conditions hold:
\begin{subequations}
\label{LIRKW:eqn:LIRKWtype2-conditions}
\begin{eqnarray}
\label{LIRKW:eqn:LIRKtype2-condition-T}
\sum_j \Phi^2_j(\tau) = \frac{1}{\gamma(\tau)} \quad \forall\, \tau \in T ~~ \mbox{with } \rho(\tau) \le p\,, \\
\label{LIRKW:eqn:LIRKtype2-condition-TW}
\sum_j \Phi^2_j(\tau) = 0 \quad \forall\, \tau \in LW_2\backslash T ~~ \mbox{with } \rho(\tau) \le p\,.
\end{eqnarray}
\end{subequations}
Here $\rho(t)$ is the number of vertices of the tree $t$, and $\gamma(t)$ is the ``product of $\rho(t)$ and all orders of the trees which appear, if the roots, one after another, are removed from $t$'' \cite[Section II.2]{Hairer_book_I}.
\end{theorem}
\begin{proof}
 The proof follows from our discussion, equation \eqref{LIRKW:eqn:lirkwexpansiontype2} and the order conditions of Rosenbrock-W methods \cite[Section IV.7]{Hairer_book_II}.
\end{proof}

\subsection{Order conditions for LIRK-W methods of type 3}

Here we derive order conditions for LIRK-W methods of type 3 \eqref{eqn:lirkwtype3}.  Making use of the expansion \eqref{LIRKW:eqn:faadibruno} we see that
\begin{subequations}
\label{LIRKW:eqn:lirkwexpansiontype3}
\begin{equation}
\left.\left(Y_i\right)^{(q)}\right|_{h=0} = q \displaystyle\sum_{j=1}^{i-1}a_{i,j} \displaystyle\sum_{m\geq 1}\frac{\partial^m F}{\partial y^m}\left(Y_j^{(\mu_1)}, \dots, Y_j^{(\mu_m)}\right) + q\displaystyle\sum_{j=1}^i \gamma_{i,j} \displaystyle\sum_{k=0}^{q-1}{q-1 \choose k}\gamma_{i,i}^k\left(\frac{d^k}{dt^k}\Lb\right)Y_j^{(q-1-k)}
\end{equation}
\begin{equation}
\left.\left(y_{n+1}\right)^{(q)}\right|_{h=0} = q \displaystyle\sum_{j=1}^s b_j \displaystyle\sum_{m>=1} \frac{\partial^m f}{\partial y^m}\left(Y_j^{(\mu_1)}, \dots, Y_j^{(\mu_m)}\right) + q\Lb \displaystyle\sum_{j=1}^s g_j \left(Y_j\right)^{(q-1)}
\end{equation}
\end{subequations}

Figures \ref{LIRKW:fig:lirkwtrees1} and \ref{LIRKW:fig:lirkwtrees2}, give the $LW_3$-trees, $\tau_i$, representing the elementary differentials, $f(\tau_i)$, present in the Taylor expansion of the exact and numerical solutions, as well as the coefficients of these terms, $\Phi^3(\tau_i)$ and $P(\tau_i)$ respectively.  The $LW_3$-trees are defined exactly as the $LW_2$-trees with only $\Phi^3(\tau)$ differing from its type 2 counterpart.  The recursive definition of $\Phi^3(\tau_i)$ is given by \eqref{LIRKW:eqn:recursiveordertype3}.

\begin{subequations}
\label{LIRKW:eqn:recursiveordertype3}
\begin{eqnarray}
\label{LIRKW:eqn:reorderphitype3}
\Phi^3_j(\tau) & = & \left\{ 
\begin{array}{lcrcl} 
b_j\bar{\Phi}^3_{j}(\tau_1)\cdots\bar{\Phi}^3_{j}(\tau_m) & \textrm{if} & \tau & = & \left[\tau_1, \dots, \tau_m\right]_{\bullet}, \\
g_j \bar{\Phi}^3_{j}(\tau_1) & \textrm{if} & \tau & = & \left[\tau_1\right]_\circ, \\
b_j & \textrm{if} & \tau & = & \left[ \; \right]_\bullet, \\
g_j & \textrm{if} & \tau & = & \left[ \; \right]_\circ,
\end{array} \right. \\
\label{LIRKW:eqn:reorderphibartype3}
\bar{\Phi}^2_j(\tau) & = & \left\{
\begin{array}{lcrcl}
a_{j,k} \bar{\Phi}^3_{k}(\tau_1) \cdots \bar{\Phi}^3_{k}(\tau_m) & \textrm{if} & \tau & = & \left[\tau_1, \dots, \tau_m\right]_\bullet, \\
\gamma_{j,k} \bar{\Phi}^3_{k}(\tau_1) & \textrm{if} & \tau & = & \left[ \tau_1 \right]_\circ, \\
\gamma_{jj}^{p}\gamma_{jk} \bar{\Phi}^3_k(\tau_1) & \textrm{if} & \tau & = & \left[ \tau_1 \right]_{\theta_p}, \\
a_{j,k} & \textrm{if} & \tau & = & \left[ \; \right]_\bullet, \\
\gamma_{j,k} & \textrm{if} & \tau & = & \left[ \; \right]_\circ, \\
\end{array} \right.
\end{eqnarray}
\end{subequations}

Similar to the reasoning for type 1 and 2 methods the order conditions for type 3 methods can be summarized by the following theorem.

\begin{theorem}[Order conditions for LIRK-W methods of type 3]\label{LIRKW:thm:LIRKWtype3-conditions}
A LIRK-W method of type 1 has order $p$ iff the following order conditions hold:
\begin{subequations}
\label{LIRKW:eqn:LIRKWtype3-conditions}
\begin{eqnarray}
\label{LIRKW:eqn:LIRKtype3-condition-T}
\sum_j \Phi^3_j(\tau) = \frac{1}{\gamma(\tau)} \quad \forall\, \tau \in T ~~ \mbox{with } \rho(\tau) \le p\,, \\
\label{LIRKW:eqn:LIRKtype3-condition-TW}
\sum_j \Phi^3_j(\tau) = 0 \quad \forall\, \tau \in LW_3\backslash T ~~ \mbox{with } \rho(\tau) \le p\,.
\end{eqnarray}
\end{subequations}
Here $\rho(t)$ is the number of vertices of the tree $t$, and $\gamma(t)$ is the ``product of $\rho(t)$ and all orders of the trees which appear, if the roots, one after another, are removed from $t$'' \cite[Section II.2]{Hairer_book_I}.
\end{theorem}
\begin{proof}
 The proof follows from our discussion, equation \eqref{LIRKW:eqn:lirkwexpansiontype3} and the order conditions of Rosenbrock-W methods \cite[Section IV.7]{Hairer_book_II}.
\end{proof}
%

%%%%%%%%%%%%%%%%%%%%%%%%%%%%%%%%%%
\section{Linear stability of LIRK-W methods}
\label{LIRKW:sec:stability}
%%%%%%%%%%%%%%%%%%%%%%%%%%%%%%%%%%

To examine the linear stability of the proposed method we solve the linear test problem
\begin{equation}
\label{LIRKW:eqn:splitlinearode}
 \frac{dy}{dt} = \Lb y + \left(\mathbf{J} - \Lb \right)y\,,~~~ t_0 \leq t \leq t_F\,, \quad y(t_0) = y_n\,; \quad  y(t) \in \R^N\, \quad \mathbf{J},\,\Lb \in \R^{N \times N}.
\end{equation}
and apply the method (\ref{LIRKW:eqn:lirkwimex}) in compact form 
\begin{subequations}
\label{LIRKW:eqn:lirkwcompact}
\begin{eqnarray}
\label{LIRKW:eqn:lirkwcompact1}
\mathbf{Y} & = & \mathds{1}\otimes y_n + \left[\mathbf{A}\otimes\left(h \mathbf{J}-h \Lb(h)\right)+\mathbf{\widehat{A}} \otimes h\Lb(h)\right]\,\mathbf{Y}, \\
\label{LIRKW:eqn:lirkwcompact2}
y_{n+1} & = & y_n + \left[ \mathbf{b}^T \otimes \left(h \mathbf{J} - h \Lb(h)\right)+\mathbf{\widehat{b}}^T \otimes h\Lb(h) \right]\,\mathbf{Y},
\end{eqnarray}
\end{subequations}
where
\[
%\begin{equation}
%\label{LIRKW:eqn:compactcoeff}
\mathbf{A} = \left[ \begin{array}{ccc} a_{1,1} & \dots & a_{1,s} \\ 
				\vdots & \ddots & \vdots \\
				a_{s,1} & \dots & a_{s,s} \end{array}\right], \quad 
\mathbf{\widehat{A}} = \left[ \begin{array}{ccc} \widehat{a}_{1,1} & \dots & \widehat{a}_{1,s} \\ 
				\vdots & \ddots & \vdots \\
				\widehat{a}_{s,1} & \dots & \widehat{a}_{s,s} \end{array}\right], \quad 
\mathbf{b} = \left[ \begin{array}{c} b_1 \\ \vdots \\ b_s \end{array} \right], \quad
\mathbf{\widehat{b}} = \left[ \begin{array}{c} \widehat{b}_1 \\ \vdots \\ \widehat{b}_s \end{array} \right]
%\end{equation}
\]
and
\[
%\label{LIRKW:eqn:compactstage}
\mathbf{Y} = \left[ \begin{array}{ccc} Y_1^T & \dots & Y_s^T \end{array} \right]^T.
\]
\subsection{Type 1 methods}
%%%%%%%%%%%%%%%%%

Application of a type 1 method \eqref{LIRKW:eqn:lirkwtype1} to \eqref{LIRKW:eqn:splitlinearode} gives:
\begin{eqnarray*}
Y_i  &=& y_n +   h\, \sum_{j=1}^{i-1}\, a_{i,j}\, (\mathbf{J}-\Lb_j)\, Y_j  +  h\,\sum_{j=1}^{i}\, \widehat{a}_{i,j}\,\Lb_j\, Y_j,\\
y_{n+1} &=&  y_n +   h\, \sum_{i=1}^s\, b_{i}\, (\mathbf{J}- \Lb_i)\, Y_i  +   h\, \sum_{i=1}^s\, \widehat{b}_{i}\, \Lb_i\, Y_i.
\end{eqnarray*}
In compact notation we have:
\begin{eqnarray*}
\mathbf{Y} & = & \mathds{1}\otimes y_n + \left[(\mathbf{A}\otimes\I_N)\, \underset{i=1, \dots, s}{\textnormal{blkdiag}}\left(h \mathbf{J}-h \Lb_i\right)+(\mathbf{\widehat{A}} \otimes \I_N)\,\underset{i=1, \dots, s}{\textnormal{blkdiag}}\left(h \Lb_i\right)\right]\,\mathbf{Y}, \\
y_{n+1} & = & y_n +\left[(\mathbf{b}^T\otimes\I_N)\, \underset{i=1, \dots, s}{\textnormal{blkdiag}}\left(h \mathbf{J}-h \Lb_i\right)+(\mathbf{\widehat{b}}^T \otimes \I_N)\,\underset{i=1, \dots, s}{\textnormal{blkdiag}}\left(h \Lb_i\right)\right]\,\mathbf{Y}.
\end{eqnarray*}
Solving for $y_{n+1}$ in terms of $y_n$ leads to the transfer matrix:
\begin{eqnarray*}
y_{n+1} & = & R(h \mathbf{J},h \Lb_1,\dots,h \Lb_s)\, y_n, \\
R(h \mathbf{J},h \Lb_1,\dots,h \Lb_s) & = & \I +\left[(\mathbf{b}^T\otimes\I_N)\, \underset{i=1, \dots, s}{\textnormal{blkdiag}}\left(h \mathbf{J}-h \Lb_i\right)+(\mathbf{\widehat{b}}^T \otimes \I_N)\,\underset{i=1, \dots, s}{\textnormal{blkdiag}}\left(h \Lb_i\right)\right]\cdot \\
&& \left[\I - (\mathbf{A}\otimes\I_N)\, \underset{i=1, \dots, s}{\textnormal{blkdiag}}\left(h \mathbf{J}-h \Lb_i\right) - (\mathbf{\widehat{A}} \otimes \I_N)\,\underset{i=1, \dots, s}{\textnormal{blkdiag}}\left(h \Lb_i\right)  \right]^{-1}\,(\mathds{1} \otimes \I).
\end{eqnarray*}
For a stiffly accurate method with 
\begin{equation}
\label{LIRKW:eqn:stiff-accuracy}
 \mathbf{b}^T = \mathbf{e}_s^T\mathbf{A} \qquad \textnormal{and} \qquad \mathbf{\widehat{b}}^T = \mathbf{e}_s^T \mathbf{\widehat{A}},
\end{equation}
standard calculations give:
\begin{eqnarray*}
R(h \mathbf{J},h \Lb_1,\dots,h \Lb_s) & = & (\mathbf{e}_s^T\otimes\I) \cdot  \left[\I - (\mathbf{A}\otimes\I_N)\, \underset{i=1, \dots, s}{\textnormal{blkdiag}}\left(h \mathbf{J}-h \Lb_i\right) - (\mathbf{\widehat{A}} \otimes \I_N)\,\underset{i=1, \dots, s}{\textnormal{blkdiag}}\left(h \Lb_i\right)  \right]^{-1}\,(\mathds{1} \otimes \I).
\end{eqnarray*}
This matrix goes to zero for stiff linear parts:
\[
\Vert h \Lb_i \Vert \to \infty, ~~ i=1,\dots,s \quad \Rightarrow \quad R(h \mathbf{J},h \Lb_1,\dots,h \Lb_s) \to 0.
\]

%%%%%%%%%%%%%%%%%
\subsection{Type 2 methods}
%%%%%%%%%%%%%%%%%

Application of a type 2 method \eqref{LIRKW:eqn:lirkwtype2} to \eqref{LIRKW:eqn:splitlinearode} gives: 
\begin{eqnarray*}
Y_i  &=& y_n +   h\, \sum_{j=1}^{i-1}\, a_{i,j}\,   F(Y_j)  +  h\,\Lb \sum_{j=1}^{i-1}\, \gamma_{i,j}\,\, Y_j + h\gamma_{i,i}\Lb_i Y_i, \\
y_{n+1} &=&  y_n +   h\, \sum_{i=1}^s\, b_i\,   F(Y_i)  +   h\, \Lb\, \sum_{i=1}^s\, g_i\,  Y_i.
\end{eqnarray*}
In compact form:
\begin{eqnarray*}
\mathbf{Y} & = & \mathds{1}\otimes y_n + \left[\mathbf{A}\otimes\left(h \mathbf{J}-h \Lb\right)+\mathbf{\widehat{A}} \otimes h\Lb+ \underset{i=1, \dots, s}{\textnormal{blkdiag}}\big(\gamma_{i,i}(h\Lb_i-h\Lb)\big) \right]\,\mathbf{Y}, \\
y_{n+1} & = & y_n + \left[ \mathbf{b}^T \otimes \left(h \mathbf{J} - h \Lb\right)+\mathbf{\widehat{b}}^T \otimes h\Lb \right]\,\mathbf{Y},
\end{eqnarray*}
where $\textnormal{blkdiag}$ denotes a block-diagonal matrix with the blocks given by its arguments; the block indices should be clear from the context.
The transfer matrix reads:
\begin{eqnarray*}
R(h \mathbf{J},h \Lb_1,\dots,h \Lb_s) & = & \I +\left[(\mathbf{b}^T\otimes\I_N)\, \big( \I_s \otimes (h \mathbf{J}-h \Lb) \big)+(\mathbf{\widehat{b}}^T \otimes \I_N)\,\big( \I_s \otimes h \Lb \big)\right]\cdot \\
&& \left[\I - (\mathbf{A}\otimes\I_N)\, \big( \I_s \otimes (h \mathbf{J}-h \Lb) \big) - (\mathbf{\widehat{A}} \otimes \I_N)\,\big( \I_s \otimes h \Lb \big) \right. \\
& & \quad \left.-\underset{i=1, \dots, s}{\textnormal{blkdiag}}\big(\gamma_{i,i}(h\Lb_i-h\Lb)\big) \vphantom{\widehat{A}} \right]^{-1}\,(\mathds{1} \otimes \I).
\end{eqnarray*}
For a stiffly accurate method \eqref{LIRKW:eqn:stiff-accuracy} standard calculations give:
\begin{eqnarray*}
R(h \mathbf{J},h \Lb_1,\dots,h \Lb_s) & = & (\mathbf{e}_s^T\otimes\I) \cdot  
 \left[\I  -\underset{i=1, \dots, s}{\textnormal{blkdiag}}\big(\gamma_{i,i}(h\Lb_i-h\Lb)\big)  \right] \\
&& \cdot \left[\I - (\mathbf{A}\otimes\I_N)\, \big( \I_s \otimes (h \mathbf{J}-h \Lb) \big) - (\mathbf{\widehat{A}} \otimes \I_N)\,\big( \I_s \otimes h \Lb \big) \right.\\ 
&& \quad \left. \vphantom{\widehat{A}} -\underset{i=1, \dots, s}{\textnormal{blkdiag}}\big(\gamma_{i,i}(h\Lb_i-h\Lb)\big)  \right]^{-1}\,(\mathds{1} \otimes \I) \\
& = & (\mathbf{e}_s^T\otimes\I) \cdot  
 \left[\I  -\underset{i=1, \dots, s}{\textnormal{blkdiag}}\big(\gamma_{i,i}(h\Lb_i-h\Lb)\big)  \right] \\
&& \cdot \Big[\I - \mathbf{A}\otimes\left(h \mathbf{J}\right) - \textnormal{tril}(\boldsymbol{\Gamma}) \otimes \left(h \Lb\right) -\underset{i=1, \dots, s}{\textnormal{blkdiag}}\big(\gamma_{i,i}(h\Lb_i-h\Lb)\big)  \Big]^{-1}\,(\mathds{1} \otimes \I) \\
%& = & (\mathbf{e}_s^T\otimes\I)\cdot\left(\left[\I - \underset{i=1, \dots, s}{\textnormal{blkdiag}}\big(\gamma_{i,i}(h\Lb_i-h\Lb)\big)\right]^{-1}\left[\I - \underset{i=1, \dots, s}{\textnormal{blkdiag}}\big(\gamma_{i,i}(h\Lb_i-h\Lb)\big)\right]\right. \\
%&& - \left.\left[\I - \underset{i=1, \dots, s}{\textnormal{blkdiag}}\big(\gamma_{i,i}(h\Lb_i-h\Lb)\big)\right]^{-1}\left[\mathbf{A}\otimes\left(h \mathbf{J}\right) + \textnormal{tril}(\boldsymbol{\Gamma}) \otimes \left(h \Lb\right) \right]\right)^{-1}(\mathds{1} \otimes \I) \\
%& = & (\mathbf{e}_s^T\otimes\I)\cdot\left( \I - \underbrace{\left[\I - \underset{i=1, \dots, s}{\textnormal{blkdiag}}\big(\gamma_{i,i}(h\Lb_i-h\Lb)\big)\right]^{-1}\left[\mathbf{A}\otimes\left(h \mathbf{J}\right) + \textnormal{tril}(\boldsymbol{\Gamma}) \otimes \left(h \Lb\right) \right]}_{=0}\right)^{-1}(\mathds{1} \otimes \I) \\
%& = & (\mathbf{e}_s^T\otimes\I)\left(\I\right)(\mathds{1} \otimes \I). 
& = & (\mathbf{e}_s^T\otimes\I) \cdot  
  \\
&& \cdot \Big[\I - \Big( \mathbf{A}\otimes\left(h \mathbf{J}\right) + \textnormal{tril}(\boldsymbol{\Gamma}) \otimes \left(h \Lb\right) \Big)\,\big[\I  -\underset{i}{\underset{i=1, \dots, s}{\textnormal{blkdiag}}}\big(\gamma_{i,i}(h\Lb_i-h\Lb)\big)  \big]^{-1}  \Big]^{-1}\,(\mathds{1} \otimes \I),
\end{eqnarray*}
where we denote
\[
 \textnormal{tril}(\boldsymbol{\Gamma}) = \boldsymbol{\Gamma} - \underset{1=1,\dots,s}{\textnormal{diag}}\big(\gamma_{i,i}\big).
\]

The stability matrix goes to zero when $\Lb_i \approx \Lb$ in the sense that the difference increases slower than the matrices with increasing stiffness:
\begin{equation*}
\frac{\Vert h\Lb \Vert}{\Vert h\Lb - h\Lb_i \Vert} \to \infty, ~~ i=1,\dots,s \quad \Rightarrow  \quad R(h \mathbf{J},h \Lb_1,\dots,h \Lb_s) \to 0.
\end{equation*}
or
\begin{equation*}
\frac{\Vert h\mathbf{J} \Vert}{\Vert h\Lb - h\Lb_i \Vert} \to \infty, ~~ i=1,\dots,s \quad \Rightarrow  \quad R(h \mathbf{J},h \Lb_1,\dots,h \Lb_s) \to 0.
\end{equation*}

Alternatively, the stability matrix goes to one when the time dependent linear term $h\L_i$ grows more quickly than either $h\J$ or $h\L$ with increasing stiffness, so that
\begin{equation*}
\frac{\Vert h\Lb \Vert + \Vert h\mathbf{J} \Vert}{\Vert h\Lb_i \Vert} \to 0,~~ i=1,\dots,s \quad \Rightarrow \quad R(h \mathbf{J},h \Lb_1,\dots,h \Lb_s) \to 1.
\end{equation*}
From this it is clear that the stability of the method is dependent on the choice of both $\Lb$ and $\Lb_i$.

%%%%%%%%%%%%%%%%%%%%%%%%%%%%%%%%%%%%%%
\section{Construction of Practical LIRK-W Methods}
\label{LIRKW:sec:construction}
%%%%%%%%%%%%%%%%%%%%%%%%%%%%%%%%%%%%%%

%%%%%%%%%%%%%%%%%%%%%%%%%%%%%%%%%%%%%%
\subsection{A third-order LIRK-W method of type 1}
\label{LIRKW:sec:construction-1}
%%%%%%%%%%%%%%%%%%%%%%%%%%%%%%%%%%%%%%
For stability considerations we seek methods which are stiffly accurate \eqref{LIRKW:eqn:stiff-accuracy}, i.e.,
\begin{equation}
\label{LIRKW:eqn:stiffaccurate}
b_i = a_{s,i} \quad \textnormal{and} \quad \widehat{b}_i = \widehat{a}_{s,i}
\quad \textnormal{for} \quad i=1,\dots,s.
\end{equation}
Moreover we wish to make use of an approximate matrix factorization, specifically the alternating directions approximation.  Our choice of general form \eqref{LIRKW:eqn:lirkwtype1} and the structure of $\widetilde{\Lb}_i$ in equation \eqref{LIRKW:eqn:Ltilde} naturally leads to the choice that
\begin{equation}
\label{LIRKW:eqn:diaggam}
\widehat{a}_{i,i} = \displaystyle\sum_{j=1}^{i-1} a_{i,j}.
\end{equation}
Finally, for convenience we require that 
\begin{equation}
\label{LIRKW:eqn:cischat}
\displaystyle\sum_{j=1}^{i-1} a_{i,j} = \displaystyle\sum_{j = 1}^i \widehat{a}_{i,j}
\end{equation}

The requirements imposed by (\ref{LIRKW:eqn:stiffaccurate}), (\ref{LIRKW:eqn:diaggam}), and (\ref{LIRKW:eqn:cischat}) taken together imply that a method with general form (\ref{LIRKW:eqn:lirkw}) must have the additional following properties:
\begin{equation}
\label{LIRKW:eqn:lirkwreq}
\displaystyle\sum_{j = 1}^i \gamma_{i,j} = 0, \quad \gamma_{i,i} = \displaystyle\sum_{j = 1}^{i-1}a_{i,j}, \quad \displaystyle\sum_{j = 1}^{i-1} \gamma_{i,j} = -\gamma_{i,i}, \quad \displaystyle\sum_{i = 1}^s g_i = 0, \quad g_i = \gamma_{s,i}.
\end{equation}
Interestingly, the first of these conditions, implied by (\ref{LIRKW:eqn:cischat}), automatically satisfies all order conditions coming from trees in figures \ref{LIRKW:fig:lirkwtrees1} and \ref{LIRKW:fig:lirkwtrees2} which end in a fat node not directly preceded by a square node.  Additionally, the simplifications in equations \eqref{LIRKW:eqn:stiffaccurate}, \eqref{LIRKW:eqn:diaggam}, \eqref{LIRKW:eqn:cischat}, and \eqref{LIRKW:eqn:lirkwreq} means that several order conditions reduce to the same as those coming from other trees, so that:
\begin{equation*}
\Phi(\tau_5) = \Phi(\tau_7), \quad \Phi(\tau_{14}) = \Phi(\tau_{23}), \quad \Phi(\tau_{15}) = \Phi(\tau_{20}), \quad \Phi(\tau_{22}) = \Phi(\tau_{19}).
\end{equation*}

In this way number of order conditions required to construct such a method of order three reduces from the original twenty-three to a much more reasonable nine.  Figure \ref{LIRKW:fig:lirkwtreered} shows these trees and the order conditions associated with them.

These conditions, with significant simplification coming from equations (\ref{LIRKW:eqn:stiffaccurate}), (\ref{LIRKW:eqn:diaggam}), (\ref{LIRKW:eqn:cischat}), and (\ref{LIRKW:eqn:lirkwreq}) as well as some algebraic manipulation are given for a five-stage method as:
\begin{subequations}
\label{LIRKW:eqn:lirkwequations}
\begin{eqnarray}
 \label{LIRKW:eqn:lirkwa}
 a_{5,1} + a_{5,2} + a_{5,3} + a_{5,4} & = & 1\\
 a_{5,2}a_2 + a_{5,3}a_3 + a_{5,4}a_4  & = & \frac{1}{2}\\
 \gamma_{5,2}a_2 + \gamma_{5,3}a_3 + \gamma_{5,4}a_4 & = & -1 \\
 a_{5,2}a_2^2 + a_{5,3}a_3^2 + a_{5,4}a_4^2 & = & \frac{1}{3}\\
 a_{5,3}a_{3,2}a_2 + a_{5,4}a_{4,3}a_3 + a_{5,4}a_{4,2}a_2 & = & \frac{1}{6}\\
 a_{5,3}\gamma_{3,2}a_2 + a_{5,4}\gamma_{4,3}a_3 + a_{5,4}\gamma_{4,2}a_2 & = & -\frac{1}{3} \\
 \gamma_{5,3}\gamma_{3,2}a_2 + \gamma_{5,4}\gamma_{4,3}a_3 + \gamma_{5,4}\gamma_{4,2}a_2 & = & 1 \\
 \gamma_{5,3}a_{3,2}a_2 + \gamma_{5,4}a_{4,3}a_3 + \gamma_{5,4}a_{4,2}a_2 & = & -\frac{1}{2}\\
 \gamma_{5,2}a_2^2 + \gamma_{5,3}a_3^2 + \gamma_{5,4}a_4^2 & = & -1
\end{eqnarray}
\end{subequations}

\begin{figure}[htp]
\begin{center}
\def\arraystretch{1.5}
\footnotesize
%\begin{tabular}{|c|c|c|rcl|}
\begin{tabular}{|>{\centering\arraybackslash}m{1in}|>{\centering\arraybackslash}m{1in}|>{\centering\arraybackslash}m{1in}|>{\centering\arraybackslash}m{0.6in}>{\centering\arraybackslash}m{0.1in}>{\centering\arraybackslash}m{0.2in}|}
\hline
$\tau_{1}$ & 
\begin{tikzpicture}[scale=.5]
      \meagrenode (j) at (0,0) [label=right:$j$] {};
  \end{tikzpicture} & $f^J$ & $\sum b_j$ & = & 1 \\
\hline
$\tau_{3}$ & 
\begin{tikzpicture}[scale=.5]
      \meagrenode (j) at (0,0) [label=right:$j$] {};
      \meagrenode (k) at (1,1) [label=right:$k$] {};
      \draw[-] (j) -- (k);
  \end{tikzpicture} & $f^J_K f^K$ & $\sum b_j a_{j,k} $ & = & 1/2 \\
\hline
$\tau_{5}$ & 
\begin{tikzpicture}[scale=.5]
      \fatnode (j) at (0,0) [label=right:$j$] {};
      \meagrenode (k) at (1,1) [label=right:$k$] {};
      \draw[-] (j) -- (k);
  \end{tikzpicture} & $\Lb_{JK}f^K$ & $\sum g_j a_{j,k}$ & = & 0 \\
\hline
$\tau_{8}$ & 
\begin{tikzpicture}[scale=.5]
      \meagrenode (j) at (1,0) [label=right:$j$] {};
      \meagrenode (k) at (0,1) [label=left:$k$] {};
      \meagrenode (l) at (2,1) [label=right:$l$] {};
      \draw[-] (j) -- (k);
      \draw[-] (j) -- (l);
  \end{tikzpicture} & $f^J_{KL}f^Kf^L$ & $\sum b_j a_{j,k}a_{j,l}$ & = & 1/3 \\
\hline
$\tau_{11}$ & 
\begin{tikzpicture}[scale=.5]
      \meagrenode (j) at (0,0) [label=right:$j$] {};
      \meagrenode (k) at (1,1) [label=right:$k$] {};
      \meagrenode (l) at (0,2) [label=left:$l$] {};
      \draw[-] (j) -- (k);
      \draw[-] (k) -- (l);
  \end{tikzpicture} & $f^J_K f^K_L f^L$ & $\sum b_j a_{j,k}a_{k,l}$ & = & 1/6 \\
\hline
$\tau_{14}$ & 
\begin{tikzpicture}[scale=.5]
      \meagrenode (j) at (0,0) [label=right:$j$] {};
      \fatnode (k) at (1,1) [label=right:$k$] {};
      \meagrenode (l) at (0,2) [label=left:$l$] {};
      \draw[-] (j) -- (k);
      \draw[-] (k) -- (l);
  \end{tikzpicture} & $f^J_K \Lb_{KL}f^L$ & $\sum b_j \gamma_{j,k} a_{k,l}$ & = & 0 \\
\hline
$\tau_{15}$ & 
\begin{tikzpicture}[scale=.5]
      \fatnode (j) at (0,0) [label=right:$j$] {};
      \fatnode (k) at (1,1) [label=right:$k$] {};
      \meagrenode (l) at (0,2) [label=left:$l$] {};
      \draw[-] (j) -- (k);
      \draw[-] (k) -- (l);
  \end{tikzpicture} & $\Lb_{JK}\Lb_{KL}f^L$ & $\sum g_j \gamma_{j,k}a_{k,l}$ & = & 0 \\
\hline
$\tau_{16}$ & 
\begin{tikzpicture}[scale=.5]
      \fatnode (j) at (0,0) [label=right:$j$] {};
      \meagrenode (k) at (1,1) [label=right:$k$] {};
      \meagrenode (l) at (0,2) [label=left:$l$] {};
      \draw[-] (j) -- (k);
      \draw[-] (k) -- (l);
  \end{tikzpicture} & $\Lb_{JK}f^K_Lf^L$ & $\sum g_j a_{j,k} a_{k,l}$ & = & 0 \\
\hline
$\tau_{22}$ & 
\begin{tikzpicture}[scale=.5]
      \fatnode (j) at (0,0) [label=right:$j$] {};
      \boxnode (k) at (1,1) [label=right:$k$] {};
      \fatnode (l) at (0,2) [label=left:$l$] {};
      \draw[-] (j) -- (k);
      \draw[-] (k) -- (l);
  \end{tikzpicture} & $\Lb_{JK}\Lb'_{KL}y^L$ & $\sum g_j\gamma_{j,j}^2$ & = & 0 \\
\hline
\end{tabular}
\caption{Butcher trees and LIRK-W conditions up to order three for a type 1 method.}
\label{LIRKW:fig:lirkwtreered}
\end{center}
\end{figure}

Table \ref{LIRKW:table:lirkwcoef} gives the coefficients for a type 1 LIRK-W method of third order.  

\begin{table}
\begin{equation*}
\begin{array}{rcl}
a & = & \begin{bmatrix} 	
0 				& 					&			 		&	 				& \\ 
0.520300000000000 	& 0					& 					& 					& \\
0.026500000000000	& 0.938000000000000	& 0					&					& \\
0.122175553766880	& 0.105600000000000	& 0.018300000000000	& 0					& \\
-0.033950868284890& 0.218016324016351	& 0.258600000000000	& 0.557334544268539	& 0\end{bmatrix} \\
\\
\gamma & = & \begin{bmatrix}
0				&					&					&					& \\
-0.520300000000000& 0.520300000000000	& 					&					& \\
0.911500000000000 & -1.876000000000000	& 0.964500000000000	& 					& \\
-0.401069249711528& 0.663393695944647	& -0.508400000000000	& 0.246075553766880	& \\
-0.155925222099085& -0.084089256959580	& -1.070724285228281	& 0.310738764286946	& 1 \end{bmatrix} \\
\\
b & = & \begin{bmatrix}  -0.033950868284890& 0.218016324016351	& 0.258600000000000	& 0.557334544268539	& 0\end{bmatrix} \\
\\
g & = & \begin{bmatrix} -0.155925222099085& -0.084089256959580	& -1.070724285228281	& 0.310738764286946	& 1 \end{bmatrix}
\end{array}
\end{equation*}
\caption{Method coefficients for a third order LIRK-W method of type 1.}
\label{LIRKW:table:lirkwcoef}
\end{table}

%%%%%%%%%%%%%%%%%%%%%%%%%%%%%%%%%%%%%%
\subsection{Third-order LIRK-W methods of type 2}
\label{LIRKW:sec:construction-2}
%%%%%%%%%%%%%%%%%%%%%%%%%%%%%%%%%%%%%%

For stability considersations we again seek a method which is stiffly accurate, and so satisfies \eqref{LIRKW:eqn:stiffaccurate}, as well as the condition that \eqref{LIRKW:eqn:cischat} to reduce the number of required order conditions.  Due to the difference in general forms we no longer impose the condition \eqref{LIRKW:eqn:diaggam}, which provides the freedom to impose the condition that
\begin{equation}
\label{LIRKW:eqn:gammadiag}
\gamma_{i,i} = \left\{ \begin{array}{rcl} 0 & \textrm{for} & i = 1 \\ 
									\gamma & \textrm{for} & i = 2, \dots, s \end{array} \right.
\end{equation}

Taking requirements imposed by \eqref{LIRKW:eqn:stiffaccurate}, \eqref{LIRKW:eqn:cischat}, and \eqref{LIRKW:eqn:gammadiag} taken together imply that a method with general form \eqref{LIRKW:eqn:lirkwtype2} must have the additional following properties:
\begin{equation}
\label{LIRKW:eqn:lirkwreqtype2}
\displaystyle\sum_{j = 1}^i \gamma_{i,j} = 0, \quad \displaystyle\sum_{j = 1}^{i-1} \gamma_{i,j} = -\gamma_{i,i}, \quad \displaystyle\sum_{i = 1}^s g_i = 0, \quad g_i = \gamma_{s,i}.
\end{equation}
Once again the first of these conditions, implied by (\ref{LIRKW:eqn:cischat}), automatically satisfies all order conditions coming from trees in figures \ref{LIRKW:fig:lirkwtrees1} and \ref{LIRKW:fig:lirkwtrees2} which end in a fat node not directly preceded by a square node.  In this way the number of required order conditions required to construct such a method of order three reduces from the original nineteen to a much more reasonable ten. Figure \ref{LIRKW:fig:lirkwtreeredtype2} shows these trees and the order conditions associated with them.

\begin{figure}[htp]
\begin{center}
\def\arraystretch{1.5}
\footnotesize
\begin{tabular}{|>{\centering\arraybackslash}m{1in}|>{\centering\arraybackslash}m{1in}|>{\centering\arraybackslash}m{1in}|>{\centering\arraybackslash}m{0.6in}>{\centering\arraybackslash}m{0.1in}>{\centering\arraybackslash}m{0.2in}|}
\hline
$\tau_1$ & 
\begin{tikzpicture}[scale=.5]
      \meagrenode (j) at (0,0) [label=right:$j$] {};
  \end{tikzpicture} & $f^J$ & $\sum b_j$ & = & 1 \\
\hline
$\tau_3$ & 
\begin{tikzpicture}[scale=.5]
      \meagrenode (j) at (0,0) [label=right:$j$] {};
      \meagrenode (k) at (1,1) [label=right:$k$] {};
      \draw[-] (j) -- (k);
  \end{tikzpicture} & $f^J_K f^K$ & $\sum b_j a_{j,k} $ & = & 1/2 \\
\hline
$\tau_5$ & 
\begin{tikzpicture}[scale=.5]
      \fatnode (j) at (0,0) [label=right:$j$] {};
      \meagrenode (k) at (1,1) [label=right:$k$] {};
      \draw[-] (j) -- (k);
  \end{tikzpicture} & $\Lb_{JK}f^K$ & $\sum g_j a_{j,k}$ & = & 0 \\
\hline
$\tau_8$ & 
\begin{tikzpicture}[scale=.5]
      \meagrenode (j) at (1,0) [label=right:$j$] {};
      \meagrenode (k) at (0,1) [label=left:$k$] {};
      \meagrenode (l) at (2,1) [label=right:$l$] {};
      \draw[-] (j) -- (k);
      \draw[-] (j) -- (l);
  \end{tikzpicture} & $f^J_{KL}f^Kf^L$ & $\sum b_j a_{j,k}a_{j,l}$ & = & 1/3 \\
\hline
$\tau_{11}$ & 
\begin{tikzpicture}[scale=.5]
      \meagrenode (j) at (0,0) [label=right:$j$] {};
      \meagrenode (k) at (1,1) [label=right:$k$] {};
      \meagrenode (l) at (0,2) [label=left:$l$] {};
      \draw[-] (j) -- (k);
      \draw[-] (k) -- (l);
  \end{tikzpicture} & $f^J_K f^K_L f^L$ & $\sum b_j a_{j,k}a_{k,l}$ & = & 1/6 \\
\hline
$\tau_{14}$ & 
\begin{tikzpicture}[scale=.5]
      \meagrenode (j) at (0,0) [label=right:$j$] {};
      \fatnode (k) at (1,1) [label=right:$k$] {};
      \meagrenode (l) at (0,2) [label=left:$l$] {};
      \draw[-] (j) -- (k);
      \draw[-] (k) -- (l);
  \end{tikzpicture} & $f^J_K \Lb_{KL}f^L$ & $\sum b_j \gamma_{j,k} a_{k,l}$ & = & 0 \\
\hline
$\tau_{15}$ & 
\begin{tikzpicture}[scale=.5]
      \fatnode (j) at (0,0) [label=right:$j$] {};
      \fatnode (k) at (1,1) [label=right:$k$] {};
      \meagrenode (l) at (0,2) [label=left:$l$] {};
      \draw[-] (j) -- (k);
      \draw[-] (k) -- (l);
  \end{tikzpicture} & $\Lb_{JK}\Lb_{KL}f^L$ & $\sum g_j \gamma_{j,k}a_{k,l}$ & = & 0 \\
\hline
$\tau_{16}$ & 
\begin{tikzpicture}[scale=.5]
      \fatnode (j) at (0,0) [label=right:$j$] {};
      \meagrenode (k) at (1,1) [label=right:$k$] {};
      \meagrenode (l) at (0,2) [label=left:$l$] {};
      \draw[-] (j) -- (k);
      \draw[-] (k) -- (l);
  \end{tikzpicture} & $\Lb_{JK}f^K_Lf^L$ & $\sum g_j a_{j,k} a_{k,l}$ & = & 0 \\
\hline
$\tau_{18}$ & 
\begin{tikzpicture}[scale=.5]
      \fatnode (j) at (0,0) [label=right:$j$] {};
      \boxnode (k) at (1,1) [label=right:$k$] {};
      \fatnode (l) at (0,2) [label=left:$l$] {};
      \draw[-] (j) -- (k);
      \draw[-] (k) -- (l);
  \end{tikzpicture} & $\Lb_{JK}\Lb'_{KL}y^L$ & $g_j\gamma_{j,j}^2$ & = & 0 \\
\hline
$\tau_{19}$ &
\begin{tikzpicture}[scale=.5]
      \meagrenode (j) at (0,0) [label=right:$j$] {};
      \boxnode (k) at (1,1) [label=right:$k$] {};
      \fatnode (l) at (0,2) [label=left:$l$] {};
      \draw[-] (j) -- (k);
      \draw[-] (k) -- (l);
  \end{tikzpicture} & $f^J_K\Lb'_{KL}y^L$ & $b_j \gamma_{j,j}^2$ & = & 0 \\
\hline
\end{tabular}
\caption{Butcher trees and LIRK-W conditions up to order three for a type 2 method.}
\label{LIRKW:fig:lirkwtreeredtype2}
\end{center}
\end{figure}

Examining the order conditions coming from the first and last two trees of figure \ref{LIRKW:fig:lirkwtreeredtype2} leads to the additional strategic choice that 
\begin{equation}
\label{LIRKW:eqn:strategicchoice}
a_{5,1} = 1, \quad \gamma_{5,1} = 0
\end{equation}
so that the order conditions coming from $\tau_{18}$ and $\tau_{19}$ are automatically satisfied.

These conditions, with significant simplification coming from equations (\ref{LIRKW:eqn:stiffaccurate}),  (\ref{LIRKW:eqn:cischat}), (\ref{LIRKW:eqn:lirkwreqtype2}), and \eqref{LIRKW:eqn:strategicchoice} as well as some algebraic manipulation are given for a five-stage method as:

\begin{subequations}
\label{LIRKW:eqn:lirkwequationstype2}
\begin{eqnarray}
a_{5,2} + a_{5,3} + a_{5,4} & = & 0 \\
\gamma_{5,2} + \gamma_{5,3} + \gamma_{5,4} & = & -\gamma \\
a_{5,2}a_2 + a_{5,3}a_3 + a_{5,4}a_4 & = & \frac{1}{2} \\
\gamma_{5,2}a_2 + \gamma_{5,3}a_3 + \gamma_{5,4}a_4 & = & -\gamma \\
a_{5,2}a_2^2 + a_{5,3}a_3^2 + a_{5,4}a_4^2 & = & \frac{1}{3} \\
a_{5,3}a_{3,2}a_2 + a_{5,4}a_{4,2}a_2 + a_{5,4}a_{4,3}a_3 & = & \frac{1}{6} \\
a_{5,3}\gamma_{3,2}a_2 + a_{5,4}\gamma_{4,2}a_2 + a_{5,4}\gamma_{4,3}a_3 & = & -\frac{1}{2}\gamma \\
\gamma_{5,3}\gamma_{3,2}a_2 + \gamma_{5,4}\gamma_{4,2}a_2 + \gamma_{5,4}\gamma_{4,3}a_3 & = & \gamma^2 \\
\gamma_{5,3}a_{3,2}a_2 + \gamma_{5,4}a_{4,2}a_2 + \gamma_{5,4}a_{4,3}a_3 & = & -\frac{1}{2}\gamma
\end{eqnarray}
\end{subequations}

\begin{table}
\begin{equation*}
\begin{array}{rcl}
a & = & \begin{bmatrix} 	
0 		&				&			 		&	 				& \\ 
\frac{1}{6} 	& 0				& 					& 					& \\
\frac{1}{3} - \frac{9 \gamma +2 \gamma_{5,4}}{3 (5 \gamma +2 \gamma_{5,4})}	& \frac{9 \gamma +2 \gamma_{5,4}}{3 (5 \gamma +2 \gamma_{5,4})}	& 0	&	& \\
\frac{1}{2} - \frac{2}{3}(3a_{4,3}-1) - a_{4,3}	& \frac{2}{3}(3a_{4,3}-1)	& a_{4,3} & 0					& \\
1 &-\frac{3}{2}	& 0	& \frac{3}{2}	& 0\end{bmatrix} \\
\\
\gamma & = & \begin{bmatrix}
0				&					&					&					& \\
-\gamma& \gamma	& 					&					& \\
\frac{2 \left(3 \gamma ^2+\gamma \gamma_{5,4}\right)}{5 \gamma +2\gamma_{5,4}}-\gamma & -\frac{2 \left(3 \gamma ^2+\gamma \gamma_{5,4}\right)}{5 \gamma +2\gamma_{5,4}}& \gamma	& & \\
2 (\gamma +\gamma_{4,3})-\gamma-\gamma_{4,3}& -2 (\gamma +\gamma_{4,3})	& \gamma_{4,3}	& \gamma	& \\
0& 4\gamma +\gamma_{5,4}	& -5\gamma-2\gamma_{5,4},	& \gamma_{5,4}	& \gamma \end{bmatrix} \\
\\
b_i & = & \begin{bmatrix} 1 &-\frac{3}{2}	& 0	& \frac{3}{2}	& 0 \end{bmatrix} \\
g_i & = & \begin{bmatrix} 0& 4\gamma +\gamma_{5,4}	& -5\gamma-2\gamma_{5,4},	& \gamma_{5,4}	& \gamma \end{bmatrix}
\end{array}
\end{equation*}
\caption{Method coefficients for a third order LIRK-W method of type 2.}
\label{LIRKW:table:lirkwcoef2}
\end{table}

%%%%%%%%%%%%%%%%%%%%%%%%%%%%%%%%%%%%%%
\section{Numerical Experiments}
\label{LIRKW:sec:numerics}
%%%%%%%%%%%%%%%%%%%%%%%%%%%%%%%%%%%%%%

We show here the convergence of both type 1 and type 2 LIRK-W schemes when applied to the shallow water equations on a sphere.  The shallow water equations in spherical coordinates are
\begin{subequations}
\begin{eqnarray}
\frac{\partial u}{\partial t} + \frac{1}{a \cos\theta} \left(u\frac{\partial u}{\partial \lambda} + v \cos\theta\frac{\partial u}{\partial \theta}\right) - \left(f + \frac{u\tan\theta}{a}\right)a + \frac{g}{a\cos\theta}\frac{\partial h}{\partial \lambda} & = & 0 \\
\frac{\partial v}{\partial t} + \frac{1}{a\cos\theta}\left(u\frac{\partial v}{\partial \lambda} + v\cos\theta\frac{\partial v}{\partial \theta}\right) + \left(f + \frac{u\tan\theta}{a}\right)u + \frac{g}{a}\frac{\partial h}{\partial\theta} & = & 0 \\
\frac{\partial h}{\partial t} + \frac{1}{a\cos\theta} \left(\frac{\partial (hu)}{\partial \lambda} + \frac{\partial (hv\cos\theta)}{\partial\theta}\right) & = & 0.
\end{eqnarray}
\end{subequations}
Where $f = 2\Omega\sin\theta$, $h$ is the height of the atmosphere, $u$ is the zonal wind component, $v$ is the meridional wind component, $\theta$ and $\lambda$ are the latitudinal and longitudinal directions, $a$ is the radius of the earth, $\Omega$ is the rotational velocity of the earth, and $g$ is the gravitational constant. The space discretization is performed using the unstaggered Turkel-Zwas scheme \cite{Navon_1987_swe, Navon_1991_swe}, with 72 nodes in the longitudinal direction and 36 nodes in the latitudinal direction.  Figure \ref{LIRKW:fig:sweconv} confirms the third order convergence of type 1 and type 2 LIRK-W schemes applied to the spherical shallow water equations, when $\Lb = \J$.

%\textcolor{red}{Can we show speed up for shallow water using a directional splitting of the Jacobian? That would be the main goal of our work - faster solves for large problems.}

%%%%%%%%%%%%%%%%%%%%%%%%
\begin{figure}[h]
\centering
\includegraphics[width=4in]{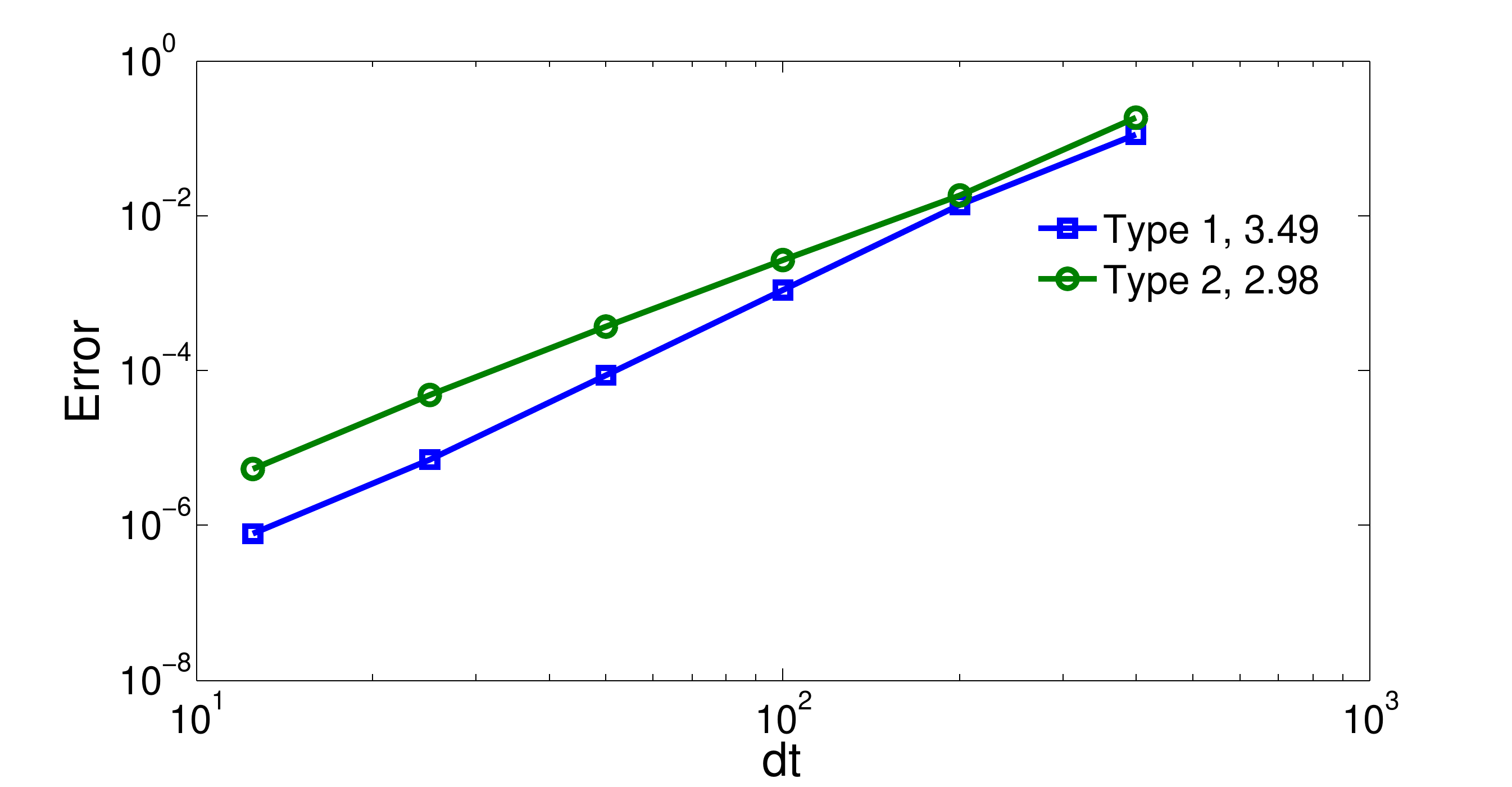}\\
\caption{Convergence diagram for LIRK-W methods applied to spherical shallow water equations.}
\label{LIRKW:fig:sweconv}
\end{figure}
%%%%%%%%%%%%%%%%%%%%%%%%

The numerical exploration of these schemes is ongoing, and work is currently being done to apply these methods to large computational fluid dynamics models.  This exploration will include the application of several different formulations of the time dependent, and stage varying linear term $\Lb_i y$.

%\textcolor{red}{Need more experiments with different $\Lb_i$.  Allen-Cahn?  Something else?  What approximations to use?  Derive a method for GMRES, and make use of it?}

%%%%%%%%%%%%%%%%%%%%%%%%%%%%%%%%%%%%%%
\section{Conclusions}
\label{LIRKW:sec:conclusions}
%%%%%%%%%%%%%%%%%%%%%%%%%%%%%%%%%%%%%%

We have presented a new class of time-integration schemes, which permit an arbitrary, time dependent, and stage varying linear approximation of the stiff system dynamics, called Linearly-Implicit Runge-Kutta-W methods.  These schemes maintain high-order when making use of an approximate matrix factorization, and can be reformulated to permit other time dependent approximations.  Future work that explores new formulations of the method and order conditions to permit new, or different, approximations may lead to more efficient integrators.

%%%%%%%%%%%%%%%%%%%%%%%%%%%%%%%%%%%%%%
\section*{References}
%%%%%%%%%%%%%%%%%%%%%%%%%%%%%%%%%%%%%%
%\textcolor{red}{Please have correct and consistent capitalization everywhere, e.g., turkel should be Turkel, Solving Ordinary .. should be Solving ordinary ...,  etc.}

\bibliographystyle{siam}
\bibliography{Bib/Master,Bib/ode_krylov,Bib/pde_time_implicit,Bib/ode_general,Bib/ode_multirate,Bib/ode_exponential,Bib/ode_rosenbrock,Bib/sandu}

\begin{thebibliography}{10}

\bibitem{Akrivis_2004}
{\sc G.~Akrivis and M.~Crouzeix}, {\em Linearly implicit methods for nonlinear
  parabolic equations}, Mathematics of Computation, 73 (2004), pp.~pp.
  613--635.

\bibitem{Akrivis_2003_IMEX}
{\sc G.~Akrivis, O.~Karakashian, and F.~Karakatsani}, {\em Linearly implicit
  methods for nonlinear evolution equations}, Numerische Mathematik, 94 (2003),
  pp.~403--418.

\bibitem{Ascher_1997}
{\sc U.~Ascher, S.~Ruuth, and R.~Spiteri}, {\em Implicit-explicit
  {R}unge-{K}utta methods for time-dependent partial differential equations},
  Applied Numerical Mathematics, 25 (1997), pp.~151--167.

\bibitem{Calvo_2001}
{\sc M.~Calvo, J.~de~Frutos, and J.~Novo}, {\em Linearly implicit
  {R}unge-{K}utta methods for advection-reaction-diffusion equations}, Applied
  Numerical Mathematics, 37 (2001), pp.~535--549.

\bibitem{Douglas_1962}
{\sc J.~{Douglas Jr.}}, {\em Alternating direction methods for three space
  variables}, Numerische Mathematik, 4 (1962), pp.~41--63.

\bibitem{Grooms_2011}
{\sc I.~Grooms and K.~Julien}, {\em Linearly implicit methods for nonlinear
  {PDEs} with linear dispersion and dissipation}, Journal of Computational
  Physics, 230 (2011), pp.~3630 -- 3650.

\bibitem{Hairer_book_I}
{\sc E.~Hairer, S.~Norsett, and G.~Wanner}, {\em Solving Ordinary Differential
  Equations {I}: Nonstiff Problems}, Springer, 2008.

\bibitem{Hairer_book_II}
{\sc E.~Hairer and G.~Wanner}, {\em Solving Ordinary Differential Equations
  {II}: Stiff and Differential-Algebraic Problems}, Springer, 2002.

\bibitem{Navon_1991_swe}
{\sc I.~Navon and J.~Yu}, {\em Exshall: A turkel-zwas explicit large time-step
  fortran program for solving the shallow-water equations in spherical
  coordinates}, Computers \& Geosciences, 17 (1991), pp.~1311 -- 1343.

\bibitem{Navon_1987_swe}
{\sc I.~M. Navon and R.~de~Villiers}, {\em The application of the turkel–zwas
  explicit large time-step scheme to a hemispheric barotropic model with
  constraint restoration}, Monthly Weather Review, 115 (1987), pp.~1036--1052.

\bibitem{Peaceman_1955}
{\sc D.~Peaceman and H.~{Rachford Jr.}}, {\em The numerical solution of
  parabolic and elliptic differential equations}, Journal of the Society for
  Industrial and Applied Mathematics, 3 (1955), pp.~28--41.

\bibitem{Rang_2005_ROW3}
{\sc J.~Rang and L.~Angermann}, {\em New {R}osenbrock {W}-methods of order 3
  for partial differential algebraic equations of index 1}, BIT Numerical
  Mathematics, 45 (2005), pp.~761--787.

\bibitem{Weiner_1997_rowmap}
{\sc R.Weiner, B.~Schmitt, and H.~Podhaisky}, {\em {ROWMAP}--a {ROW}-code with
  {K}rylov techniques for large stiff {ODE}s}, Applied Numerical Mathematics,
  25 (1997), pp.~303--319.

\bibitem{Saad}
{\sc Y.~Saad}, {\em Iterative methods for sparse linear systems}, PWS Pub. Co.,
  Boston, 1996.

\bibitem{Sandu_2015_GARK}
{\sc A.~Sandu and M.~G\"{u}nther}, {\em A generalized-structure approach to
  additive {Runge-Kutta} methods}, SIAM Journal on Numerical Analysis, 53
  (2015), pp.~17--42.

\bibitem{Tranquilli2016}
{\sc P.~Tranquilli, S.~R. Glandon, A.~Sarshar, and A.~Sandu}, {\em Analytical
  jacobian-vector products for the matrix-free time integration of partial
  differential equations}, Journal of Computational and Applied Mathematics,
  (2016), pp.~--.

\bibitem{Tranquilli_2014_ROK}
{\sc P.~Tranquilli and A.~Sandu}, {\em Rosenbrock-krylov methods for large
  systems of differential equations}, SIAM J. Scientific Computing, 36 (2014).

\end{thebibliography}
\end{document}